\documentclass[11pt,reqno]{amsart}
\usepackage{graphicx}
\usepackage{color}
\usepackage{mathrsfs,amsmath,amssymb,amsthm,amsfonts}

\usepackage{graphics,url,color,epsfig}
\usepackage{tikz}
\usepackage{xcolor}
\definecolor{mysoftblue}{hsb}{0.55,0.15,0.9}
\definecolor{mysoftpurple}{hsb}{0.9,0.25,0.9}
\definecolor{mysoftorange}{hsb}{0.15,0.3,0.95}
\usetikzlibrary{patterns}
\usepackage[colorlinks]{hyperref}
\AtBeginDocument{%
  \hypersetup{%
    linkcolor=blue,%
    citecolor=blue,%
  }%
}
\usepackage{cleveref}
\usepackage{pgfplots}
\usepackage{multicol}
\usepackage{fix-cm}

\def\R{\mathbb{R}}

\makeatletter

\newcommand{\Rmnum}[1]{\expandafter\@slowromancap\romannumeral #1@}
\makeatother

\newtheorem{thm}{Theorem}[section]

\newtheorem{lemma}[thm]{Lemma}
\newtheorem{remark}[thm]{Remark}
\newtheorem{theorem}[thm]{Theorem}

\usepackage[numbers,sort&compress]{natbib}
\DeclareMathOperator{\supp}{supp}

\newcommand{\mb}{\mbox}
\everymath{\displaystyle}

\allowdisplaybreaks

\begin{document}
\author{Wenxiong Chen}
\address{Wenxiong Chen \newline\indent Department of Mathematical Sciences \newline\indent Yeshiva University \newline\indent New York, NY, 10033, USA}
\email{wchen@yu.edu}

\author{Yahong Guo}
\address{Yahong Guo \newline\indent School of Mathematical Sciences \newline\indent  Shanghai Jiao Tong University \newline\indent Shanghai, 200240, P. R. China}
\email{yhguo@sjtu.edu.cn}

\author{Leyun Wu}
\address{Leyun Wu
\newline\indent
School of Mathematics \newline\indent South China University of Technology \newline\indent Guangzhou, 510640, P. R. China}
\email{leyunwu@scut.edu.cn}

\title{Monotonicity for the fractional semi-linear problem  in a half space}

\begin{abstract}
In this paper, we study semilinear fractional equations $$(-\Delta)^s u(x) = f(u(x))$$ in a half-space and prove that all positive solutions are strictly increasing in the $x_n$-direction.

Previous results typically require the solution $u$ to be globally bounded in $\mathbb{R}^n$. We substantially weaken this condition by assuming only that $u$ be bounded in each slab. Moreover, our analysis relies solely on the local Lipschitz continuity of the nonlinearity $f$, which is weaker than the conditions imposed in earlier works.

As a crucial ingredient, we obtained a boundary H\"{older} regularity estimate that requires only the boundedness of $u$ near the boundary. This represents a significant improvement over existing results, which often assumed global boundedness of $u$ throughout $\mathbb{R}^n$. The proof introduces a new idea that may be of independent interest. 

To derive the monotonicity, we employ the method of moving planes. We first obtain a narrow region principle in unbounded domains, which ensures that the moving plane procedure can be initiated from $x_n = 0$. We then establish two averaging effects for the solutions to fractional equations. These key ingredients guarantee that the planes can be moved  continuously all the way to $x_n = \infty$. 

Previously, narrow region principle can only be applied to a single narrow region. It is for the first time that we establish a multiple narrow region principle that can be applied simultaneously to finitely many narrow regions. 

Compared with the traditional approaches, methods based on the {\em averaging effect} require substantially weaker regularity assumptions and can even accommodate unbounded solutions.

We believe that these new ideas and techniques develop here will serve as powerful tools in 
studying qualitative properties of solutions to fractional equations.

\end{abstract}



\maketitle

\numberwithin{equation}{section}
\section{Introduction and main results}

In this paper, we study the monotonicity of  solutions to the fractional semi-linear problem
\begin{equation}\label{main}
    \begin{cases}
    (-\Delta)^s u(x)=f(u), &  x\in \mathbb{R}^n_+,\\
    u(x)=0 & x \in \mathbb{R}^n\backslash \mathbb{R}^n_+,
      \end{cases}
\end{equation}
in a half space $\mathbb{R}^n_+=\{x =(x', x_n) \in
\mathbb{R}^n : x_n >0\}$
with $0<s<1$ and $n\geq 2s.$

Here $(-\Delta)^s$  is the fractional Laplacian,  a nonlocal operator defined by the singular integral
\begin{equation}\label{eq1-1}
(-\Delta)^s u(x)=C_{n, s}PV \int_{\mathbb{R}^n}\frac{u(x)-u(y)}{|x-y|^{n+2s}}dy,
\end{equation}
where $PV$
stands for the Cauchy principal value, and $C_{n, s}$ is a normalization constant. 

In order that the integral on the right hand side of \eqref{eq1-1} is well defined, we require that $u\in C_{loc}^{1, 1}(\mathbb{R}^n_+)\cap \mathcal{L}_{2s}$, where
$$
\mathcal{L}_{2s}=\left\{u\in L_{loc}^1 \,\Big| \int_{\mathbb{R}^n}\frac{|u(x)|}{1+|x|^{n+2s}}dx<\infty \right\}
$$
endowed naturally with the norm
$$
\|u\|_{\mathcal{L}_{2s}}:= \int_{\mathbb{R}^n}\frac{|u(x)|}{1+|x|^{n+2s}}dx.
$$
Because of the non-locality of  the fractional Laplacian, it is necessary to  assume $u(x)=0$ in the whole complement of $\mathbb{R}^n_+$, rather than merely on  the boundary $\partial \mathbb{R}^n_+.$

Our goal is to prove that any nonnegative, nontrivial solutions $u$ of \eqref{main} is srictly increasing in the $x_n$-direction, without assuming that $u$ is bounded in the whole space $\mathbb{R}^n$ and under very weak conditions on the nonlinearity $f(u)$. We emphasize that monotonicity in a half-space plays a crucial role in qualitative analysis, as it often serves as a key step toward classification results, nonexistence theorems, and one-dimensional symmetry.



The fractional Laplacian and, more generally, nonlocal elliptic operators arise both in pure mathematical research and in wide range of applications, such as the thin obstacle problem \cite{S2007CPAM, CSS2008IM}, minimal surfaces \cite{CRS2010CPAM, CV2011CVPDE}, phase transitions \cite{SV2009JFA}, crystal dislocation \cite{GM2012DCDS}, Markov processes \cite{GM2005SD}, and fractional quantum mechanics \cite{L2000PLA}. For an elementary introduction to the subject, see \cite{DPV2012BSM} and the references therein.

As pointed out in \cite{G2006CMP}, the fractional Laplacian is the infinitesimal generator of symmetric stable processes $(0<s<1)$, commonly known in the physics literature as L\'evy flights. These processes paly a central role in stochastic modeling, with applications in operations research, queuing theory, mathematical finance, and risk estimation. In contrast to Brownian motion $(s=1)$, which arises as the scaling limit of random walks where particles jump only to nearest-neighbor sites, stable processes can be viewed as the limiting models of random walks in which particles may jump to arbitrary sites, with probabilities decaying according to a power law (see also \cite{V2009BSEMAS}).
The essential difference is that such models incorporate long-range interactions rather than purely local ones.
\medskip

Before presenting our results, we briefly recall some known related achievements which serve as the main motivation for our study. 

 In the local case $s=1$, the most relevant references  are a series of papers by Berestycki, Caffarelli, and Nirenberg \cite{BCN1,BCN2,BCN3,BCN4}, where  several qualitative properties of solutions to
\begin{equation}\label{main-i}
    \begin{cases}
    -\Delta u(x) = f(u), & x \in \mathbb{R}^n_+, \\[0.3em]
    u(x) = 0, & x \in \partial \mathbb{R}^n_+,
    \end{cases}
\end{equation}
were established under the assumption that the nonlinearity $f$ is Lipschitz continuous.
In particular, they proved that if $f(0) \geq 0$, then every positive solution of \eqref{main-i} satisfies
\[
\frac{\partial u}{\partial x_n} > 0 \quad \text{in } \mathbb{R}^n_+
\]
(see \cite{BCN1,BCN2}). This monotonicity property was first established by Dancer in \cite{Dancer1,Dancer2} under stronger assumptions on both the solutions and the nonlinearities. Later, Farina \cite{Farina2020ME} showed that if $u$ is bounded on each slab $\mathbb{R}^{n-1} \times [0,z]$ for every $z>0$, $f(0)\geq 0$, and $f$ is locally Lipschitz, then every solution to \eqref{main-i} is strictly increasing in the $x_n$-direction.
More recently, Beuvin, Farina, and Sciunzi \cite{BFS2025} established monotonicity results for the equation $-\Delta u = f(u)$ on continuous epigraphs bounded from below, subject to homogeneous Dirichlet boundary conditions, assuming that $f$ is a (locally or globally) Lipschitz function with $f(0)\geq 0$; see also \cite{EL1982PRSE}.
The case $f(0) < 0$ is considerably more delicate and remains to date,  not completely understood. For dimensions $n=2$ or $n=3$, several results were obtained in \cite{BCN2,FS2016RMI,FS2013JMAA}, while in higher dimensions some partial results were derived in \cite{CEG2016JMPA} under the additional assumption $u \geq u_0$. The main difficulty stems from the existence of a one-dimensional periodic solutions of \eqref{main-i} that are not strictly positive.

To the best of our knowledge, only partial results are available regarding the monotonicity of solutions to the nonlocal problem \eqref{main}.

For special nonlinearities of the form $f(t) = t^p$, Quaas and Xia \cite{QX2013CVPDE} established a Liouville-type theorem for bounded positive solutions, while Chen, Fang, and Yang \cite{CFY2015} considered the case of solutions satisfying certain integrability conditions.
Dipierro, Soave, and Valdinoci \cite{DSV2017MA} analyzed monotonicity in more general domains with the epigraph property,  but only for a very restricted class of nonlinearities.

For rather more general nonlinearity $f(t)$, in \cite{FW2016CCM}, Fall and Weth established monotonicity for positive {\em bounded} solutions of \eqref{main} under the assumptions that $f$ is a non-decreasing, locally Lipschitz function satisfying $f(t)>0$ for $t>0$ and
\[
\lim_{\substack{r,t \to 0 \\ r \neq t}} \frac{f(r)-f(t)}{r-t} = 0.
\]
Subsequently, Barrios, Del Pezzo, Garc\'{i}a-Meli\'{a}n, and Quaas \cite{BDGQ2017CVPDE} proved that nonnegative, nontrivial, {\em bounded} classical solutions of \eqref{main} are strictly positive and strictly monotone increasing in the $x_n$-direction.

Most recently, Barrios, Garc\'{i}a-Meli\'{a}n, and Quaas \cite{BGQ2019PAMS} showed that the same conclusions remain valid if $f$ is assumed only to be locally Lipschitz. However, as in \cite{FW2016CCM} and \cite{BDGQ2017CVPDE}, their results still  require the solutions to be {\em globally bounded}. 

This common {\em global boundedness assumption on the solutions} appears to be mainly technical, and it is partly due to the nonlocal nature of the fractional Laplacian. 

Can this assumption be weakened? Addressing this issue is one of the motivation of our present paper.  After introducing several innovative techniques, we are able to relax the {\em global boundedness condition} and only require the solutions be bounded on each slab $\Omega_\lambda = \{x \mid 0<x_n< \lambda.\}$:

\textbf{(BDS)} For every $\lambda > 0$ there exists a constant
$C(\lambda) > 0$ such that
$0 \leq u \leq C(\lambda) $ on  $\Omega_\lambda.$
\medskip

\begin{theorem}\label{th1.10}
Let $n\geq 2s,$ {$f\in C_{loc}^{0, 1}([0, \infty))$}
 with $f(0) \geq 0$. Assume that  $u\in C_{loc}^{1, 1}(\mathbb{R}^n_+)\cap \mathcal{L}_{2s}$ 
 is a nonnegative, nontrivial classical  solution of \eqref{main} satisfying the \textbf{(BDS)}  condition.
Then 
$$ \frac{\partial u}{\partial x_n} (x) > 0, \;\; \forall \, x \in \mathbb{R}^n_+.$$
\end{theorem}

As usual, the theorem is proved by using the method of moving planes. In this procedure, one classical approach is to take limit along a sequence of solutions, which requires boundary regularity of solutions. In \cite{BGQ2019PAMS}, in order to obtain boundary H\"{o}lder regularity for the solutions $u$, the authors constructed a super solution; however due to the nonlocal nature of the fractional Laplacian, this construction 
requires $u$ to be {\em globally bounded}. 

Employing a completely new idea, we are able to establish the same boundary regularity under much weaker assumption that the solutions are merely {\em locally bounded} near the boundary. 

\begin{theorem}\label{th1.4} (Boundary regularity)
Assume that $u \in \mathcal{L}_{2s}$ is a nonnegative solution to
\begin{equation}\label{eqfp}
\begin{cases}
(-\Delta)^su(x)=f(x), & x\in \mathbb{R}^n_+,\\
u(x)=0,& x\in \mathbb{R}^n \backslash  \mathbb{R}^n_+.
\end{cases}
\end{equation}
Suppose $u(x)$ and $f(x)$ are locally  bounded near the boundary $\partial \mathbb{R}^n_+$. 

Then $u(x)$ is $s$-H\"older continuous up to the boundary $\partial \mathbb{R}^n_+$. More precisely,
\[
u(x) \leq C x_n^{s},
\qquad \text{for } x_n>0 \text{ sufficiently small},
\]
where $C$ is a constant depending only on $n$, $s$, 
the $L^\infty$ norm of $u$ and $f$ near the boundary,  and $\|u\|_{\mathcal{L}_{2s}}$.
\end{theorem}

Previous results on the boundary regularity of solutions to nonlocal equations in half-spaces can be found in \cite{FW2016CCM,DSV2017MA}, where it was shown that $u$ is $s$-H\"older continuous on the boundary of the half-space via the Green representation formula, under the assumptions 
$$u \in L^\infty(\mathbb{R}^n) \mbox{ and } f \in L^\infty(\mathbb{R}^n_+).$$

The boundary behavior of solutions to boundary value problems driven by integro-differential operators has been extensively studied \cite{B1997, G2014, G2015, ROS1, ROS2, ROS3}.
For boundary value problems of the type \eqref{eqfp}, it is often necessary to assume that $u$ is bounded throughout the entire space $\mathbb{R}^n$.

In contrast, in our Theorem~\ref{th1.4}, by employing a novel iterative technique together with a refined analysis, we establish the same boundary regularity
 under the significantly weaker assumption that $u$ and $f$ are bounded only in a neighborhood of the boundary $\partial \mathbb{R}^n_+$, rather than in the whole space $\mathbb{R}^n$. We believe that this boundary regularity theorem will find broad applications in the study of qualitative properties of solutions to fractional equations.

To establish the boundary regularity (Theorem \ref{th1.4}), we first decompose the solution into two parts: the potential function
\[
   v(x) = \int_{B_1(e_n)} G(x,y)\, f(y)\, dy,
\]
and the $s$-harmonic part
\[
   w(x) = u(x) - v(x),
\]
which can be represented by the Poisson kernel. 

We first show that $v(x)$ is $s$-H\"older continuous by constructing a suitable supersolution. 

Then under the local bounded-ness assumption on $u$,  by a careful analysis of the Poisson integral, we derive that
$w(x)$ is $\min\{s,\, 1-s\}$-H\"older continuous, 
Hence $u(x)$ is also $\min\{s,\, 1-s\}$-H\"older continuous. If $s \leq \frac{1}{2}$, we are done. Otherwise, substituting this H\"{o}lder regularity of $u$ 
back into the Poisson integral and through another round of fine estimate, we obtain higher regularity of $w$ and hence of $u$. After finitely many steps of such iterations, we ultimately arrive at $s$-H\"{o}lder boundary regularity of the solution $u$.
\medskip 

Our second main result concerning the monotonicity of the solution without assuming the global boundedness of solutions is as follows. 

\begin{theorem}\label{th1.1}
Assume $n\geq 2s,$ {$f\in C_{loc}^{0, 1}([0, \infty))$}
 with $f(0) \geq 0$ and let {$u\in C_{loc}^{1, 1}(\mathbb{R}^n_+)\cap \mathcal{L}_{2s}\cap C^0(\overline {\mathbb{R}^n_+})$ }
 be a  uniformly continuous,  nonnegative, nontrivial classical  solution of \eqref{main}.
Then $u(x)$ is strictly increasing along the $x_n$-direction. Furthermore
$$ \frac{\partial u}{\partial x_n} (x) > 0, \;\; \forall \, x \in \mathbb{R}^n_+.$$
\end{theorem}

To overcome the difficulty caused by the nonlocality of the fractional Laplacian, we introduce a completely new approach that repeatedly exploits  the {\em averaging effects} of nonlocal operators; in particular, this mechanism is applied four times, as will be illustrated shortly.  

\begin{remark}
There are many possible choices for the nonlinearity $f$. A prototypical example is $f(t) = t - t^3$, which leads to the fractional Allen-Cahn equation, a widely studied model of phase transitions in media with long-range particle interactions (see, e.g., \cite{SV2012}). Other classical choices include $f(t) = t^p$ with $p > 1$ and $f(t) = e^t$, which give rise to the fractional Lane--Emden and fractional Liouville equations, respectively.

Our results are obtained under substantially weaker conditions on $f$.
For instance, in \cite{FW2016CCM, CM2023JFA, CHM2024ANS, CW-ANS-2021}, one
typically imposes additional hypotheses, such as $f(t) > 0$ for $t > 0$,
or $f(0) = 0$ with $f'(0) \leq 0$, or the boundedness of $f'$.

For other fractional problems on half space, please see \cite{DLL, LZ, LZ1}. 
\end{remark}


To obtain the monotonicity, the main frame of the approach is the method of moving planes. Here we list  some related notation that will be used throughout this paper.

For $\lambda \in \mathbb{R}$, let
\[
T_\lambda := \left\{ x=(x',x_n) \in \mathbb{R}^n \;\middle|\; x_n=\lambda \right\}
\]
denote the hyperplane perpendicular to the $x_n$-axis. Define
\[
\Sigma_\lambda := \{x \in \mathbb{R}^n \mid x_n < \lambda\},
\qquad
\Omega_\lambda := \{x \in \mathbb{R}^n_+ \mid x_n < \lambda\},
\]
which represent the regions under the hyperplane $T_\lambda$ in $\mathbb{R}^n$ and in $\mathbb{R}^n_+$, respectively.
The reflection of a point $x$ with respect to $T_\lambda$ is given by
\[
x^\lambda := (x_1,\dots,x_{n-1},\,2\lambda - x_n) = (x',\,2\lambda - x_n).
\]

To compare the solution $u(x)$ with its reflection $u_\lambda(x) := u(x^\lambda)$, we consider
\[
w_\lambda(x) := u_\lambda(x) - u(x).
\]
Obviously $w_\lambda$ is antisymmetric with respect to the hyperplane $T_\lambda$.
\medskip

To initiate the method of moving planes from near $x_n = 0$. the following narrow region principle for antisymmetric functions in unbounded domains is a crucial ingredient.

\begin{theorem}[Narrow region principle in a union of finitely many narrow slabs]
\label{lem:multi-slab-nrp-simple}

Let  $\Omega \subset \Sigma_\lambda$ be an open set containing in a union of finitely many disjoint narrow slabs: 
\[
\Omega \subset \bigcup_{i=1}^N \{ x \in \mathbb{R}^n : a_i < x_n < a_i + l_i \}, \quad l_i>0. 
\]

Suppose that
\begin{equation}\label{eq:multi-pde-quant-simple}
\begin{cases}
(-\Delta)^s w(x) = c(x) w(x), & x \in \Omega \cap \{ w<0 \},\\
w(x) \ge 0, & x \in \Sigma_\lambda \setminus \Omega,\\
w(x) = -w(x^\lambda), & x \in \Sigma_\lambda
\end{cases}
\end{equation}
with the growth condition
\[
w(x) \ge -C_0 (1 + |x|^\gamma), \quad 0<\gamma<2s.
\]  
Assume that $c(x)$ is bounded from above in $\Omega$.

Then there exists a sufficiently  small $l>0$, depending on $n,s,\sup_\Omega c(x)$ and the minimal distance between the slabs, such that if
\[
\max_{1 \le i \le N} l_i < l,
\]
then
\[
w(x) \ge 0 \quad \text{in } \Sigma_\lambda.
\]
Moreover, if $w(\tilde x)=0$ at some point $\tilde x \in \Sigma_\lambda$, then $w \equiv 0$ in $\mathbb{R}^n$.
\end{theorem}

\begin{remark} 
All previous narrow region principles apply only to a single narrow region. Now it  is for the first time, a multiple narrow region principle is established here, which relies on a new idea in the proof. Such a principle not only provide a starting point to move the plane in the first step, but also serve as a powerful tool in the second step. As the readers will see in the later section, it allows us to conclude immediately that the sequence under consideration stays away from both narrow boundary slabs, thereby significantly simplifying the subsequent analysis. 
\end{remark}

This theorem provides a starting point to move the plane near the boundary $x_n = 0$. We then continuous increase $\lambda$ and 
prove that the planes $T_\lambda$ can be moved all the way up 
to $\lambda = +\infty$.

Otherwise, suppose that 
$$ \lambda_0 = \sup \{ \lambda \mid w_\mu (x) \geq 0, x \in \Sigma_\mu, \; \mu \leq \lambda\} < \infty,$$
then one could construct a decreasing sequence $\{\lambda_k\}$ with $\lambda_k \to \lambda_0$ and 
$$ \inf_{\Sigma_{\lambda_k}} w_{\lambda_k}(x)  < 0.$$
To derive a contradiction, the main difficulty lies in the fact that $\Sigma_{\lambda_k}$ are unbounded, and hence the minimum of $w_{\lambda_k}$ may not
be attained. A classical approach (cf.\ \cite{Farina2020ME, CW-ANS-2021}) is to take limit along a sequence of solutions to arrive at limiting equations and 
derive a contradiction. However, for such a sequence of equations to converges, one typically requires higher regularity (at least uniformly $C^{2s+\epsilon}$) on the sequence of solutions. 

In this paper, we introduce a new and fundamentally different approach--applying the {\em averaging effects} along such sequences to obtain a contradiction. In this process, we only require the sequence be uniformly continuous and can even accommodate unbounded solutions. 

The following averaging effects on the solution $u$ and the associated antisymmetric function $w_\lambda$ play a crucial role in the second step of the method of moving planes.

\begin{theorem}\label{th1.3}({Averaging effects}) 
Let $D \subset \mathbb{R}^n_+$  be a bounded domain. Define
$$
 D_R:= \left\{ x\in \mathbb{R}^n \backslash D \,\, \mid \,\, dist(x, D) \leq R,\,\, x_n \geq dist(D, \partial \mathbb{R}^n_+)/2\right\}
$$
for any fixed $R>0.$

Assume that
$u\in C_{loc}^{1, 1}(\mathbb{R}^n)\cap \mathcal{L}_{2s}$ 
is a nonnegative uniformly continuous solution of \eqref{main} and there exists a positive constant $C_0 >0$ such that
$$
u(x) \geq C_0>0, \,\, \forall x\in D.
$$
Suppose that  $f(x)$ is continuous in $[0, \infty)$ with  $f(0)\geq 0.$

Then there exists a positive constant $\delta_0>0$ depending only on  $R$ and $C_0$  such that
\begin{equation}\label{eq1.7}
u(x) \geq \delta_0,\,\, \forall x \in D_R.
\end{equation}

However, this $\delta_0$ is {\em independent} of where $D$ is. 
\end{theorem}

\begin{remark}
 
 It is crucial  that the positive constant $\delta_0$ in the theorem is independent of the location of $D$. In the applications, this set $D$ may move toward infinity, while $\delta_0$ remains unchanged.  This is the essential difference between the averaging effect and the strong maximum principle. 

\end{remark} 

\begin{theorem}\label{AE-anti}(Averaging effects for antisymmetric functions)
Let $\tilde D \subset \Sigma_\lambda$ be a bounded domain. Define
$$
 \tilde D_R:= \left\{ x\in \mathbb{R}^n \backslash \tilde D \,\, \mid \,\, dist(x, \tilde D) \leq R,\,\, x_n \leq \lambda-dist(\tilde D, T_\lambda)/2\right\}
$$
for any fixed $R>0.$

Assume that $w\in C_{loc}^{1, 1}(\mathbb{R}^n)\cap \mathcal{L}_{2s}$ is uniformly continuous, and  satisfies
 \begin{equation}\label{eq1.8}
 \begin{cases}
 (-\Delta)^s w=c(x) w(x),& x\in \Sigma_\lambda,\\
 w(x)=-w(x^\lambda),& x\in \Sigma_\lambda\\
 w(x)\geq 0,& x\in \Sigma_\lambda,
 \end{cases}
 \end{equation}
 with $c(x)$ bounded from below. If there exists a constant  $\tilde C_0 >0$ such that
 $$
 w(x)\geq \tilde C_0 >0,\,\, x \in \tilde D,
 $$
 then there exists a constant $\tilde \delta_0>0$ depending only on $R$ and $\tilde C$ such that
\begin{equation}\label{eq1.9}
 w(x) \geq \tilde \delta_0,\,\, x \in \tilde D_R.
\end{equation}
\end{theorem}




\begin{remark}
The averaging effect asserts that if a solution to  fractional equations is positive in a region
$D$, then this positivity is averaged out, or diffused,  to any larger region containing
$D$. Moreover, the influence becomes stronger as the boundaries of the two regions get closer. These conclusions remain valid for antisymmetric functions as well.
\end{remark}

Now in the second step of the method of moving planes, if the plane $T_\lambda$ has to stop somewhere, we  would obtain a sequence $\{x^k\}$ such that $w_{\lambda_k}(x^k)$ approaches its infimum.
Applying the averaging effects twice, we deduce that $u(x^k) \to 0$ as $k \to \infty$.
Then we consider the normalized functions
\[
   v_k(x) = \frac{u(x)}{u(x_k)},
\]
along with their antisymmetric counterparts. Employing the averaging effects once more, together with boundary
regularity estimates, we reach a contradiction. 
\medskip

The remainder of the paper is organized as follows. Section~2 is devoted to the proof of boundary regularity for fractional equations (Theorem~\ref{th1.4}). Section~3 is concerned with the proofs of several key tools used throughout the paper, including the narrow region principle in unbounded domains (Theorem~\ref{lem:multi-slab-nrp-simple}), the averaging effects for fractional equations (Theorem~\ref{th1.3}), and the averaging effects for antisymmetric functions (Theorem~\ref{AE-anti}). In Section~4, we apply the method of moving planes to establish the monotonicity of solutions to \eqref{main} in the $x_n$-direction, thereby proving Theorem~ \ref{th1.10} and \ref{th1.1}.

\section{Boundary regularity}

 In what follows, we shall use $C_i, r_i(i =
 1,2,...)$  to denote  generic positive constants which may vary in the context.

In this section, we prove Theorem~\ref{th1.4} on boundary H\"older regularity for fractional equations by decomposing the solutions into two parts and applying an iterative technique.

\begin{proof}[Proof of Theorem \ref{th1.4}.]
For any $\tilde{x} \in  \mathbb{R}^n_+$, we aim to show that
\[
u(\tilde{x}) \leq C \, \tilde{x}_n^s.
\]
Denote $x_0=(x', 0) $.
Without loss of generality, we may assume that $x_0 = 0$ and
\(\tilde{x} = (0', \tilde{x}_n)\).

Define
$$
v(x)=\int_{B_1(e_n)} G(x,y) f(y)dy.
$$
Here
$$
G(x, y)=\frac{C_{n, s}}{|x-y|^{n-2s}}\frac{1}{B(s, \frac{n}{2}-s)} \int_0^{\frac{(1-|x-e_n|^2)(1-|y-e_n|^2)}{|x-y|^2}} \frac{b^{s-1}}{(b+1)^{\frac{n}{2}}} db
$$
denotes the Green function in $B_1(e_n)$.
It is easy to check that $v(x)$ satisfies
\begin{equation*}
\begin{cases}
(-\Delta)^s v(x)=f(x),& x\in B_1(e_n),\\
v(x)=0,& x \in \mathbb{R}^n \backslash B_1(e_n).
\end{cases}
\end{equation*}
Denote
$$
w(x):=u(x)-v(x).
$$
Then $w(x)$ satisfies
\begin{equation*}
\begin{cases}
(-\Delta)^s w(x)=0,& x\in B_1(e_n),\\
w(x)=u(x),& x \in \mathbb{R}^n \backslash B_1(e_n).
\end{cases}
\end{equation*}

We divide the proof into three steps.

{\it{Step 1.} } We first show that $v(x)$ is $s$-H\"{o}lder continuous near the boundary and
\begin{equation}\label{eq3-8}
|v(\tilde x)| \leq C {\tilde x_n}^s
\end{equation}
by constructing a super-solution.


Let
$$
\phi(x)=
\begin{cases}
    C_0(1-|x-e_n|^2)^s_+, & x \in B_1(e_n),\\
    0,& x\in \mathbb{R}^n \backslash B_1(e_n).
\end{cases}
$$
It is  well-known that
$$
(-\Delta)^s \phi(x) =1, \,\, x\in B_1(e_n)
$$
for a suitable constant $C_0$, and that $\phi(x)$ is H\"{o}lder continuous near $\partial B_1(e_n)$.

Define
$$
\bar v(x)=\|f\|_{L^\infty(B_1(e_n))}  \phi(x). 
$$
Then $\bar v (x)$ satisfies
\begin{eqnarray*}
\begin{cases}
(-\Delta)^s \bar v(x)=\|f\|_{L^\infty(B_1(e_n))} \geq f(x) = (-\Delta)^s v(x) & x\in B_1(e_n),\\
\bar v(x)=0 = v(x), & x\in \mathbb{R}^n \backslash B_1(e_n).
\end{cases}
\end{eqnarray*}
Now $\bar v(x)$ is a super solution. 
By the maximum principle for the fractional Laplacian, we obtain
$$
\bar v(x) \geq v(x), \,\, x \in B_1(e_n).
$$

Similarly, let  
$$
\tilde v(x)=-\|f\|_{L^\infty(B_1(e_n))}  \phi(x). 
$$
Then $\tilde v (x)$ is a sub-solution satisfying
\begin{eqnarray*}
\begin{cases}
(-\Delta)^s \tilde v(x)=-\|f\|_{L^\infty(B_1(e_n))} \leq f(x) = (-\Delta)^s v(x) & x\in B_1(e_n),\\
\tilde v(x)=0 = v(x), & x\in \mathbb{R}^n \backslash B_1(e_n).
\end{cases}
\end{eqnarray*}
It follows from the maximum principle that 
$$
\tilde v(x) \leq v(x), \,\, x \in B_1(e_n).
$$
As a consequence, 
\[|v(x)|\leq  \|f\|_{L^\infty(B_1(e_n))} \phi(x),\ \forall x\in B_1(e_n),\]
in particular,
$$
| v(\tilde x)| \leq C {\tilde x_n}^s.
$$

{\it{Step 2.}} 
Then we consider the harmonic part and show that
\begin{equation}\label{eq3-1}
w(\tilde x) \leq C \max\{{\tilde x_n}^s, {\tilde x_n}^{1-s}\}.
\end{equation}

Let
 $$
 D=\{x\in \mathbb{R}^n \mid 0<|x'|<a,\,\, 0<x_n<1-\sqrt{1-|x'|^2}\}
 $$
 with sufficiently small $a$, see Figure 1 below.

 \begin{center}
 \begin{tikzpicture}[scale=2.5]
    	\fill[white] (-1.5,-0.5) rectangle (2,2);
    	\draw[pattern=north west lines, blue!20] (-0.8,0) rectangle (0.8,1);
    	\fill[white](0,1) circle(1)   ; 	
    	\draw[black,thick,->] (-1.8,0)--(1.8,0);
    	\node[black] at (2.0,0)  {$x'$};
    	\draw[black,thick,->] (0,-0.5)--(0,2.5);
    	\node[black] at (0.3,2.5) {$x_{n}$};
    	\node[black] at (0.15,-0.1) {$\textit{O}$};
    	\fill (0,0) circle (0.02);
    	\fill[red!80!black] (0,1) circle (0.03);
    	\node[red!80!black] at (-0.2,1)[ font=\fontsize{15}{18}\selectfont] {$e_{n}$};
    	\draw [black,line width=0.7](0,1) circle (1);
    	\draw[red!80!black,thick,dashed] (0,1)--(1,1);
    	\node[red!80!black] at (0.5,1.15)[ font=\fontsize{15}{15}\selectfont] {\scriptsize$1$};
    	\node[blue!80] at (-0.75,-0.15)[ font=\fontsize{15}{18}\selectfont] {$-a$};
    	\node[blue!80] at (0.75,-0.15)[ font=\fontsize{15}{18}\selectfont] {$a$};
    	\draw[blue!70!black,thick,->] (0.7,0.15) arc(270:360:0.6 and 0.3);
    	\node[blue] at (1.3,0.58) [ font=\fontsize{15}{18}\selectfont]{$D$};
    	\fill[red!80!black] (0,0.2) circle (0.03);
            \node[red!80!black] at (-0.2,0.25) [font=\fontsize{15}{18}\selectfont] {$\tilde x$};
            \node[below=1cm, align=center, text width=8cm] at (0,-0.3)
            {\fontsize{12}{12}\selectfont Figure~1. The domain $D$.};
            \end{tikzpicture}
\end{center}


From the Poisson representation, 
\begin{equation}\label{eq3-2}
\begin{aligned}
w(\tilde x)=&\int_{|y-e_n|>1} \frac{(1-|\tilde x-e_n|^2)^s}{(|y-e_n|^2-1)^s} \frac{1}{|\tilde x-y|^n}u(y)dy\\
=&\int_D \frac{(1-|\tilde x-e_n|^2)^s}{(|y-e_n|^2-1)^s} \frac{1}{|\tilde x-y|^n}u(y)dy\\
&+\int_{\{|y-e_n|>1\}\backslash D} \frac{(1-|\tilde x-e_n|^2)^s}{(|y-e_n|^2-1)^s} \frac{1}{|\tilde x-y|^n}u(y)dy\\
:=&I_1+I_2.
\end{aligned}
\end{equation}
Considering the term $I_2$ in \eqref{eq3-2}, and observing that $(1-|\tilde x-e_n|^2)^s \sim {\tilde x_n}^s$, we arrive at
\begin{equation}\label{eq3-3}
    I_2 \sim {\tilde x_n}^s \int_{\mathbb{R}^n} \frac{u(y)}{1+|y|^{n+2s}}
dy \sim {\tilde x_n}^s.
\end{equation}
As $u$ is bounded in $D$, the term $I_1$ can be estimated as follows:
\begin{equation}\label{eq3-4}
\begin{aligned}
I_1 = & \int_D \frac{(1-|\tilde x-e_n|^2)^s}{(|y-e_n|^2-1)^s} \frac{1}{|\tilde x-y|^n}u(y)dy\\
\leq & C \int_D \frac{(1-|\tilde x-e_n|^2)^s}{(|y-e_n|^2-1)^s} \frac{1}{|\tilde x-y|^n}dy\\
\sim &  {\tilde x_n}^s \int_0^a \int_0^{1-\sqrt{1-r^2}} \frac{1}{[r^2+y_n^2-2y_n]^s [r^2+(y_n-{\tilde x_n})^2]^\frac{n}{2}} dy_n r^{n-2} dr\\
=&{\tilde x_n}^s \int_0^a \int_0^{\frac{1-\sqrt{1-r^2}}{r}} \frac{r^{-1-2s}}{
(1+z^2-\frac{2z}{r})^s[1+(z-\frac{{\tilde x_n}}{r})^2]^{\frac{n}{2}}
} dzdr\\
=&{\tilde x_n}^s \int_0^{{\tilde x_n}}\int_0^{\frac{1-\sqrt{1-r^2}}{r}} \frac{r^{-1-2s}}{
(1+z^2-\frac{2z}{r})^s[1+(z-\frac{{\tilde x_n}}{r})^2]^{\frac{n}{2}}
}  dzdr\\
&+{\tilde x_n}^s \int_{{\tilde x_n}}^a \int_{0}^{\frac{1-\sqrt{1-r^2}}{r}} \frac{r^{-1-2s}}{
(1+z^2-\frac{2z}{r})^s[1+(z-\frac{{\tilde x_n}}{r})^2]^{\frac{n}{2}}
}  dzdr\\
:=&I_{11}+I_{12}.
\end{aligned}
\end{equation}
For the term $I_{11}$, since $r\leq {\tilde x_n}$ and   $a$ is sufficiently small, we have
$$
[1+(z-\frac{{\tilde x_n}}{r})^2]^{\frac{n}{2}}\sim \left(\frac{{\tilde x_n}}{r}\right)^n.
$$
As a consequence,
\begin{equation}\label{eq3-5}
\begin{aligned}
I_{11}\sim &  {\tilde x_n}^{s-n}  \int_0^{{\tilde x_n}}  \int_{0}^{\frac{1-\sqrt{1-r^2}}{r}}  \frac{r^{n-2s}  }{(1+z^2-\frac{2z}{r})^s} dzdr\\
\sim & {\tilde x_n}^{s-n}  \int_0^{{\tilde x_n}}  \int_0^ 1 \frac{r^{n-1-2s}  }{t^s} dtdr\\
\sim & {\tilde x_n}^{1-s}.
\end{aligned}
\end{equation}
For the term $I_{12},$ noting that for $0<z<\frac{1-\sqrt{1-r^2}}{r}=\frac{r}{1+\sqrt{1-r^2}}<r$, we have
$$
(1+z^2-\frac{2z}{r})^{-s}=\left(\frac{1-\sqrt{1-r^2}}{r}-z\right)^{-s}\left(\frac{1+\sqrt{1-r^2}}{r}-z\right)^{-s}\leq Cr^s \left(\frac{1-\sqrt{1-r^2}}{r}-z\right)^{-s}.
$$
Therefore,
\begin{equation}\label{eq3-6}
\begin{aligned}
I_{12} \leq & {\tilde x_n}^s \int_{{\tilde x_n}}^a \int_{0}^{\frac{1-\sqrt{1-r^2}}{r}} \frac{r^{-1-2s}}{
(1+z^2-\frac{2z}{r})^s} dzdr\\
\leq  & C {\tilde x_n}^s \int_{{\tilde x_n}}^a \int_{0}^{\frac{1-\sqrt{1-r^2}}{r}} r^{-1-s} \left(\frac{1-\sqrt{1-r^2}}{r}-z\right)^{-s}dzdr\\
\leq & C {\tilde x_n}^s \int_{{\tilde x_n}}^a r^{-2s} dr \\
\leq & C {\tilde x_n}^{1-s}.
\end{aligned}
\end{equation}
Combining \eqref{eq3-2}, \eqref{eq3-3}, \eqref{eq3-4}, \eqref{eq3-5} with \eqref{eq3-6}, we arrive at \eqref{eq3-1}.

{\it{Step 3.}}
In this step, we show that
\begin{equation}\label{eq3-9}
u(\tilde x) \leq C\, {\tilde x_n}^s.
\end{equation}

From Steps 1 and 2, we have already established that
\[
w(\tilde x) \leq C \max\{{\tilde x_n}^s, {\tilde x_n}^{1-s}\}, \quad \text{and} \quad v(\tilde x) \leq C\, {\tilde x_n}^s,
\]
by \eqref{eq3-1} and \eqref{eq3-8}.

If \(0 < s \leq \frac{1}{2}\), we can immediately deduce that
\[
u(\tilde x) = w(\tilde x) + v(\tilde x) \leq C\, {\tilde x_n}^s,
\]
which confirms \eqref{eq3-9}.

If $\frac{1}{2} < s < 1$, then $1-s <s$ and we have ${\tilde x_n}^s < {\tilde x_n}^{1-s}$.
It follows from \eqref{eq3-1} and \eqref{eq3-8} that
\begin{equation}\label{eq3-10}
u(\tilde x) = w(\tilde x) + v(\tilde x) \leq C\, {\tilde x_n}^{1-s},
\end{equation}
which shows that $u$ is $(1-s)$-H\"older continuous near the boundary.

 By  \eqref{eq3-2} and \eqref{eq3-3}, we have
\begin{equation}\label{eq3-11}
\begin{aligned}
w(\tilde x)=&\int_{|y-e_n|>1} \frac{(1-|\tilde x-e_n|^2)^s}{(|y-e_n|^2-1)^s} \frac{1}{|\tilde x-y|^n}u(y)dy\\
\sim &\int_D \frac{(1-|\tilde x-e_n|^2)^s}{(|y-e_n|^2-1)^s} \frac{1}{|\tilde x-y|^n}u(y)dy+ C {\tilde x_n}^s\\
:= &J+ C{\tilde x_n}^s,
\end{aligned}
\end{equation}
where
 $$
 D=\{x\in \mathbb{R}^n \mid 0<|x'|<a,\,\, 0<x_n<1-\sqrt{1-|x'|^2}\}
 $$
 with sufficiently small $a>0$.

Since
$1-\sqrt{1-r^2}=\frac{r^2}{1+\sqrt{1-r^2}}<r^2 < a^2, $ and $a$ is sufficiently small, it follows from \eqref{eq3-10} that
$$
u(y) \leq C y_n^{1-s},\,\, \mbox{for}\,\, y \in D.
$$
Consequently, the term $J$ in \eqref{eq3-11} can be estimated as follows:
\begin{equation}\label{eq3-12}
\begin{aligned}
J = & \int_D \frac{(1-|\tilde x-e_n|^2)^s}{(|y-e_n|^2-1)^s} \frac{1}{|\tilde x-y|^n}u(y)dy\\
\leq & C \int_D \frac{(1-|\tilde x-e_n|^2)^s}{(|y-e_n|^2-1)^s} \frac{y_n^{1-s}}{|\tilde x-y|^n}dy\\
\sim &  {\tilde x_n}^s \int_0^a \int_0^{1-\sqrt{1-r^2}} \frac{y_n^{1-s}}{[r^2+y_n^2-2y_n]^s [r^2+(y_n-{\tilde x_n})^2]^\frac{n}{2}} dy_n r^{n-2} dr\\
=&{\tilde x_n}^s \int_0^a \int_0^{\frac{1-\sqrt{1-r^2}}{r}} \frac{r^{-3s}z^{1-s}}{
(1+z^2-\frac{2z}{r})^s[1+(z-\frac{{\tilde x_n}}{r})^2]^{\frac{n}{2}}
} dzdr\\
\leq & {\tilde x_n}^s \int_0^a \int_0^{\frac{1-\sqrt{1-r^2}}{r}} \frac{r^{-3s}}{
(1+z^2-\frac{2z}{r})^s[1+(z-\frac{{\tilde x_n}}{r})^2]^{\frac{n}{2}}
} dzdr\\
=&{\tilde x_n}^s \int_0^{{\tilde x_n}}\int_0^{\frac{1-\sqrt{1-r^2}}{r}} \frac{r^{-3s} }{
(1+z^2-\frac{2z}{r})^s[1+(z-\frac{{\tilde x_n}}{r})^2]^{\frac{n}{2}}
}  dzdr\\
&+{\tilde x_n}^s \int_{{\tilde x_n}}^a \int_{0}^{\frac{1-\sqrt{1-r^2}}{r}} \frac{r^{-3s}}{
(1+z^2-\frac{2z}{r})^s[1+(z-\frac{{\tilde x_n}}{r})^2]^{\frac{n}{2}}
}  dzdr\\
:=&J_1+J_2.
\end{aligned}
\end{equation}
By an argument similar to that used in deriving \eqref{eq3-5} and \eqref{eq3-6}, we obtain
\[
J_1 \sim {\tilde x_n}^{2-2s}, \quad \text{and} \quad J_2 \leq C\, {\tilde x_n}^{2-2s},
\]
which, together with \eqref{eq3-12}, yields
\[
J \leq C\, {\tilde x_n}^{2-2s}.
\]
Therefore,
\begin{equation}\label{eq3-13}
w(\tilde x) \leq C\, {\tilde x_n}^{2-2s}, \quad \text{if } \frac{1}{2} < s < 1.
\end{equation}

If \(\frac{1}{2} < s \leq \frac{2}{3}\), then
\[
w(\tilde x) \leq C\, {\tilde x_n}^{2-2s} \leq C\, {\tilde x_n}^s,
\]
which, together with \eqref{eq3-8}, implies \eqref{eq3-9}.

Next, we consider the case \(\frac{2}{3} < s < 1\).
By \eqref{eq3-8} and \eqref{eq3-13}, we have
\begin{equation}\label{eq3-14}
u(y) \leq C\, y_n^{2-2s}, \quad \text{for } y \in D.
\end{equation}
This allows us to obtain a more precise estimate for the term \(J\) in \eqref{eq3-11}.
Utilizing \eqref{eq3-14} and repeating the same process as in deriving \eqref{eq3-12}, we get
\[
J \leq C\, {\tilde x_n}^{3-3s}, \quad \text{hence} \quad w(\tilde x) \leq C\, {\tilde x_n}^{3-3s}.
\]
It follows that \eqref{eq3-9} holds if \(s \leq \frac{3}{4}\).

Repeating this process iteratively, we obtain
\[
w(\tilde x) \leq C\, {\tilde x_n}^{4-4s},
\]
which shows that \eqref{eq3-9} is valid if \(s \leq \frac{4}{5}\).
Continuing in this way, we finally arrive at \eqref{eq3-9} for all \(0 < s < 1\).

This completes the proof of Theorem \ref{th1.4}.
\end{proof}

\section{The narrow region principle and  averaging effects}

In this section, we first establish the narrow region principle for antisymmetric functions in unbounded domains (Theorem~\ref{lem:multi-slab-nrp-simple}). We then derive the averaging effects for solutions to \eqref{main} (Theorem~\ref{th1.3}), as well as for antisymmetric functions (Theorem~\ref{AE-anti}).

We note that Chen and Wu \cite{CW-ANS-2021} established a narrow region principle in unbounded domains for fractional parabolic equations,  see also Chen and Ma \cite{CM2023JFA} for a proof in the context of dual fractional parabolic equations.  More recently, the core methodology pertaining to the fractional elliptic case has been systematically expounded in the comprehensive survey \cite{CDG2024DCDS} by Chen, Dai and Guo. All these narrow region principles apply only to a single narrow region. In contrast, it is for the first time, a multiple narrow region principle is established here. A new idea is involved in the proof. Such a principle not only provide a starting point to move the plane, but also serve as a powerful tool in the second step in the method of moving planes. As the readers will see in the next section, it allows us to conclude immediately that the sequence under consideration stays away from both narrow boundary slabs, thereby significantly simplifying the subsequent analysis.

\begin{proof}[Proof of Theorem \ref{lem:multi-slab-nrp-simple}.]
 For each slab $\Omega_i,\,i=1, \cdots, N$, define the one-dimensional barrier
\[
\phi_i(x_n) := \Big(1 - \frac{(x_n - c_i)^2}{r_i^2}\Big)_+^s, \quad r_i = l_i/2, \quad c_i = a_i + l_i/2,
\]
and the auxiliary function
\[
h(x) := \Big(1 + \sum_{i=1}^N \phi_i(x_n) \Big) (1+|x'|^2)^{\gamma}, \quad 0< \gamma<2s.
\]
Define
\[
\bar w(x) := \frac{w(x)}{h(x)}.
\]
By the growth condition on $w$ and the choice of $\gamma$, one has
\[
\liminf_{|x| \to \infty} \bar w(x) \ge 0.
\]

 Suppose the conclusion of the theorem fails. Then there exists
\[
x^0 \in \Omega \quad \text{such that} \quad \bar w(x^0) = \inf_{\Sigma_\lambda} \bar w < 0.
\]
Without loss of generality, we assume $x^0 \in \Omega_{i_0}$ for some $1\leq i_0\leq N$.

 Decompose
\[
h(x) = (1+|x'|^2)^{\beta/2} \Big(1 + \phi_{i_0}(x_n) + \sum_{j \neq i_0} \phi_j(x_n) \Big).
\]
It is obvious that 
\[
(-\Delta)^s \phi_{i_0}(x^0_n) = \frac{C_1}{l_{i_0}^{2s}}.
\]
For $j\neq i_0$, since $x^0_n \notin \supp \phi_j$, we have
\[
|(-\Delta)^s \phi_j(x^0_n)|=C_2 |(-\Delta)^s_{\mathbb{R}^1} \phi_j(x^0_n)|  \le  \frac{C_2 l_j}{\mathrm{dist}^{\,1+2s}(\Omega_{i_0},\Omega_j)}.
\]
Since $l_k$ is sufficiently small for each $k=1, \cdots, N$, we may assume that 
$$\min \{\mathrm{dist}(\Omega_{i_0},\Omega_j)\}>\min_{k=1, \cdots, N} l_k\quad \mbox{and}\quad 
\sum_{j\neq i_0}\frac{C_2}{\mathrm{dist}^{\,2s}(\Omega_{i_0},\Omega_j)}\leq \frac{C_1}{2l
^{2s}_{i_0}}.$$
As a consequence, 
$$
\begin{aligned}
(-\Delta)^s  \Big(1 + \phi_{i_0}(x_n^0) + \sum_{j \neq i_0} \phi_j(x_n^0) \Big)
=& (-\Delta)^s  \phi_{i_0}(x_n^0) + \sum_{j \neq i_0} (-\Delta)^s\phi_j(x_n^0) \\
\geq &  \frac{C_1}{l_{i_0}^{2s}}-\sum_{j \neq i_0}\frac{C_2 l_j}{\mathrm{dist}^{\,1+2s}(\Omega_{i_0},\Omega_j)}\\
\geq &  \frac{C_1}{l_{i_0}^{2s}}-\sum_{j \neq i_0}\frac{C_2}{\mathrm{dist}^{\,2s}(\Omega_{i_0},\Omega_j)}\\
\geq &  \frac{C_1}{2l_{i_0}^{2s}}.
\end{aligned}
$$
Hence, for sufficiently small $l_{i_0}$, by a similar argument as in \cite[Lemma 2.1]{CW-ANS-2021} yields 
\[
(-\Delta)^s h(x^0) \ge \frac{C}{l_{i_0}^{2s}} h(x^0).
\]

To proceed with the proof, combining the definition of $\bar w(x)$ with the facts that $|x^0 - y| < |x^0 - y^\lambda|$ and $h(y) > h(y^\lambda)$ for $y \in \Sigma_\lambda$ gives rise to
\begin{eqnarray}\label{ineq1}
\begin{aligned}
(-\Delta)^s w(x^0)=&(-\Delta)^s  \left(\bar w(x^0) h(x^0)\right)\\
=&\bar w(x^0)(-\Delta)^s h(x^0)+C_{n, s}    P.V. \int_{\mathbb{R}^n} \frac{h(y)(\bar w(x^0)-\bar w(y))}{|x^0-y|^{n+2s}} dy \\
\leq &\bar w(x^0) (-\Delta)^s h(x^0) + C_{n, s} \int_{\Sigma_\lambda} \frac{h(y) \bar w(x^0)- w(y)}{|x^0-y|^{n+2s}} dy \\
&+ C_{n, s} \int_{\Sigma_\lambda} \frac{h(y^\lambda) \bar w(x^0)+ w(y)}{|x^0-y^{\lambda}|^{n+2s}} dy\\
\leq & \bar w(x^0) (-\Delta)^s h(x^0) + C_{n, s} \int_{\Sigma_\lambda} \frac{2h(y^\lambda) \bar w(x^0)}{|x^0-y|^{n+2s}} dy \\
\leq &\bar w(x^0) (-\Delta)^s h(x^0)\\
\leq & \frac{C}{l_{i_0}^{2s}} h(x^0) \bar w(x^0)\\
\leq & \frac{C}{l_{i_0}^{2s}}   w(x^0),
\end{aligned}
\end{eqnarray}
Substituting \eqref{ineq1} into \eqref{eq:multi-pde-quant-simple} leads to
\begin{eqnarray}\label{ineq2}
c(x^0)w (x^0)\leq  \frac{C}{l_{i_0}^{2s}}  w(x^0).
\end{eqnarray}
Since $w(x^0) = h(x^0) \bar w(x^0) < 0$ and $c(x)$ is bounded from above in $\Omega$, we derive a contradiction from \eqref{ineq2} for sufficiently small $l_{i_0}$.
Therefore, we obtain
\begin{equation*}
w(x) \ge 0 \quad \text{in } \Sigma_\lambda.
\end{equation*}
It follows from the above inequality that if there exists a point $\tilde x \in \Sigma_\lambda$ such that 
\[
w(\tilde x) = \min_{\Sigma_\lambda} w(x) = 0,
\]
then, provided that $w(x) \not\equiv 0$ in $\Sigma_\lambda$, we deduce that 
\[
(-\Delta)^s w(\tilde x) = C_{n,s} \, \text{P.V.} \int_{\Sigma_\lambda} \left( \frac{1}{|\tilde x - y^\lambda|^{\,n+2s}} - \frac{1}{|\tilde x - y|^{\,n+2s}} \right) dy < 0,
\]
which contradicts the fact that
\[
(-\Delta)^s w(\tilde x) = c(\tilde x) w(\tilde x) = 0.
\]
As  a consequence,  $w \equiv 0$ in $\mathbb{R}^n$.
This completes  the proof of Theorem \ref{lem:multi-slab-nrp-simple}.
\end{proof}

\begin{proof}[Proof of Theorem \ref{th1.3}.] 
For any $x^0\in D_R$,
 we  construct a sub-solution $\underline{u}$  in $B_\delta(x^0) \subset \R_{+}^{n}\backslash \bar D$ (see Figure 3),
where $0<\delta <\frac{1}{2}dist(D,\partial\R_{+}^{n})$ is a constant to be determined later.

Let
\[
\psi(x) = a (\delta^2 - |x|^2)_+^s.
\]
It is straightforward to verify that $\psi(x)$ satisfies
\[
(-\Delta)^s \psi(x) = 1, \quad x \in B_\delta(0),
\]
for a suitable choice of the constant $a$.
We then define a translated version of $\psi$ as follows:
\begin{equation*}
\phi(x):=\psi(x-x^0).
\end{equation*}
It is clear that $\phi(x)$ satisfies
 $$
(-\Delta)^s \phi(x)=1,\,\, x\in B_\delta(x^0).
$$

 \begin{center}
\begin{tikzpicture}[scale=1]
            \fill[mysoftblue] (0,3) circle (3);
        \node[blue!60] at (0.5,5) {\Large$D_{R}$};
        \draw[red!70!black,thick,dashed] (1,3)--(2.9,3.7);
        \node[red!70!black] at (1.7,3.6) {$R$};
        \draw[black] (0,3) circle (3);
        \fill[blue!5] (0,3) circle (1);
        \node[blue] at (0.25,3) {$D$};
        \fill[white] (-3,0) rectangle (3,1);
        \draw[black,thin,dashed,-] (-2.236068,1)--(2.236068,1);
        \draw[black] (0,3) circle (1);
        \draw[black,thick,solid,->] (-4.5,0) -- (4.5,0);
    	\node[black] at (4.8,0) {$x'$};
    	\draw[black,thick,->] (0,-1)--(0,6.5);
    	\node[black] at (0.35,6.5) {$x_{n}$};
         \fill[black] (0,0) circle (0.04);
    	\node[black] at (0.25,-0.25) {$\textit{O}$};
        \fill[red!70!black] (-1.5,4.7) circle (0.05);
        \node[red!70!black] at (-1.5,4.5) {$x^{0}$};
        \draw[red!70!black,thick] (-1.5,4.7) circle (0.6);
        \draw[red!70!black,thick,dashed] (-1.5,4.7)--(-1,4.8);
        \node[red!70!black] at (-1.3,4.95) {\scriptsize$\delta$};
        \draw[red!70!black,thin,<-] (-2.8,4.4) arc(270:360:0.9 and 0.3);
       \node[red!70!black] at (-3.2,4.8) {$B_{\delta}(x^{0})$};
        \draw[red!70!black,thick,<->,dashed] (1,0)--(1,1);
        \node[red!70!black] at (2.3,0.5){$\frac{dist(D,\partial\R_{+}^{n})}{2}$};
          \node[below=1cm, align=center] at (0,-1)
            {\fontsize{12}{12}\selectfont Figure~3. The domains $D$ and $D_R$.};
 \end{tikzpicture}
 \end{center}

Define
$$
\chi_D (x):=
\begin{cases}
1, & x\in D,\\
0, & x \in \mathbb{R}^n \backslash D
\end{cases}
$$
be the characteristic function of the set $D$.

Denote
$$
\underline u(x) =\chi_D (x) u(x)+ \varepsilon \phi(x),
$$
where $\varepsilon$ is a positive constant to be determined later.
For $x\in B_\delta (x^0)$,  we deduce from $u\geq C_0>0$ in $D$ that
\begin{eqnarray*}
\begin{aligned}
(-\Delta)^s \underline u(x)=& (-\Delta)^s (\chi_D(x) u(x))+\varepsilon\\
=&C_{n, s} P. V. \int_{\mathbb{R}^n} \frac{0-\chi_D(y) u(y)}{|x-y|^{n+2s}} dy +\varepsilon\\
=& C_{n, s} P. V. \int_D \frac{-u(y)}{|x-y|^{n+2s}} dy + \varepsilon\\
\leq & -C_1 +\varepsilon.
\end{aligned}
\end{eqnarray*}
Let $\varepsilon=\frac{C_1}{2}$. Then 
\begin{equation}\label{uderlineu}
(-\Delta)^s \underline u(x) \leq -\frac{C_1}{2},\quad x\in B_\delta(x^0).
\end{equation}
On the other hand, since $f(0)=0$,  $f$ is continuous, and $u$ is uniformly continuous, then there exist  constants $\delta_1, \delta_2>0$, independent of $x^0$, such that 
$$
|f(u(x^0))-f(0)|\leq  \frac{C_1}{4}, \,\,\, \mbox{if}\quad u(x^0)<\delta_2.
$$
and 
$$
|f(u(x))-f(u(x^0))|\leq \frac{C_1}{4}, \,\,\,  \mbox{if}\quad u(x^0)<\delta_2\,\,\mbox{and}\,\, x\in B_{\delta_1}(x^0)\subset D_R.
$$
As a consequence, if $u(x^0)<\delta_2$ and $x\in B_{\delta_1}(x^0)\subset  D_R$, we have 
\begin{equation}\label{estimateu}
(-\Delta)^s  u(x)=f(u(x))\geq-|f(u(x))-f(u(x^0))| -|f(u(x^0))-f(0)|\geq  -\frac{C_1}{2}.
\end{equation}
Taking $\delta:=\delta_1$, and combining \eqref{estimateu} with \eqref{uderlineu} gives rise to
\[
(-\Delta)^s u(x) \geq \; (-\Delta)^s \underline{u}(x),
\quad x \in B_{\delta_1}(x^0),
\]
which shows that $\underline{u}(x)$ is a sub-solution  in $B_{\delta_1}(x^0)\subset  D_R$.

To verify the exterior condition, we note that $u(x)$ is nonnegative in $\mathbb{R}^n$. Hence,
\[
\underline{u}(x) = \chi_D(x) u(x) \leq u(x), \qquad x \in \mathbb{R}^n \setminus B_{\delta_1}(x^0).
\]
Therefore, the sub-solution $\underline{u}(x)$ satisfies
\[
\begin{cases}
(-\Delta)^s \underline{u}(x) - (-\Delta)^s u(x) \leq 0, & x \in B_{\delta_1}(x^0), \\[4pt]
\underline{u}(x) - u(x) \leq 0, & x \in \mathbb{R}^n \setminus B_{\delta_1}(x^0).
\end{cases}
\]
By the maximum principle for the fractional Laplacian, we deduce that
\[
\underline{u}(x) \leq u(x), \qquad x \in B_{\delta_1}(x^0),
\]
and consequently,
\[
u(x) \geq \frac{C_1}{2} \phi(x), \,\,\,  \mbox{if}\quad u(x^0)<\delta_2\,\,\mbox{and}\,\, x\in B_{\delta_1}(x^0).
\]
In particular,  
\begin{equation}\label{ux0}
u(x^0) \geq \frac{C_1}{2} \phi(x^0)= \frac{C_1}{2} a \delta_1^{2s},\,\,\,\mbox{if} \,\, \,u(x^0)\leq \delta_2.
\end{equation}

Define
\[
\delta_0=\min\left\{\frac{C_1}{2}a\delta_1^{2s},\,\delta_2\right\}.
\]
We now deduce that for any $x^0\in D_R$, one has
\[
u(x^0)\ge \delta_0.
\]

Indeed, suppose by contradiction that $u(x^0)<\delta_0$.

\noindent\textbf{Case 1.} If $\frac{C_1}{2}a\delta_1^{2s}<\delta_2$, then
\[
u(x^0)<\delta_0=\frac{C_1}{2}a\delta_1^{2s}<\delta_2,
\]
which contradicts \eqref{ux0}.

\noindent\textbf{Case 2.} If $\frac{C_1}{2}a\delta_1^{2s}\ge \delta_2$, then
\[
u(x^0)<\delta_0=\delta_2\le \frac{C_1}{2}a\delta_1^{2s},
\]
which again contradicts \eqref{ux0}.

Therefore, $u(x^0)\ge \delta_0$ for all $x^0\in D_R$.
Consequently, \eqref{eq1.7} holds, and the proof of Theorem~\ref{th1.3} is completed.
\end{proof}

\begin{proof}[Proof of Theorem \ref{AE-anti}.]

For any $x^0\in \tilde D_R$,
 we construct an antisymmetric sub-solution $\underline{w}$ of $w$ in $B_\delta(x^0) \subset \tilde D_R$,
where $\delta > 0$ is a constant (to be specified later) satisfying
$\delta \leq \tfrac{1}{2}\,\mathrm{dist}(x^0, T_\lambda)$, where $0<\delta <\frac{1}{2}dist(\tilde D,\partial\R_{+}^{n})$.

   \begin{center}
\begin{tikzpicture}[scale=1]
      \fill[mysoftblue]  (0,-3) circle (3);
        \node[blue!60] at (0.7,-5) {\Large$\tilde{D}_{R}$};
        \draw[red!70!black,thick,dashed] (1,-3)--(3,-3);
        \node[red!70!black] at (2,-2.8) {$R$};
        \draw[black] (0,-3) circle (3);
         \fill[blue!4] (0,-3) circle (1);
        \node[blue] at (0.3,-3) {$\tilde{D}$};
        \draw[black] (0,3) circle (1);
        \fill[blue!4] (0,3) circle (1);
        \node[blue] at (0.35,3) {$\tilde{D}^{\lambda}$};
        \fill[white] (-3,-1) rectangle (3,0);
        \draw[black,thin,dashed,-] (-2.236068,-1)--(2.236068,-1);
        \draw[black] (0,-3) circle (1);
        \draw[black,thick,solid,-] (-4.5,0) -- (4.5,0);
    	\node[black] at (4.75,0) {$T_{\lambda}$};
    	\draw[black,thick,->] (0,-6.5)--(0,4.5);
    	\node[black] at (0.35,4.5) {$x_{n}$};
        \fill[red!70!black] (-1.5,-1.8) circle (0.05);
        \node[red!70!black] at (-1.55,-1.6) {$x^{k}$};
        \draw[red!70!black,thick] (-1.5,-1.8) circle (0.6);
        \draw[red!70!black,thick,dashed] (-1.5,-1.8)--(-0.9,-1.8);
        \node[red!70!black] at (-1.2,-2) {\scriptsize$\delta$};
        \draw[red!70!black,thin,->] (-1.95,-2.05) arc(0:90:0.9 and 0.3);
        \node[red!70!black] at (-3.5,-1.9) {$B_{\delta}(x^{k})$};
        \draw[red!70!black,thick,<->,dashed] (1,-1)--(1,0);
        \node[red!70!black] at (2.3,-0.5){$\frac{dist(\tilde{D},T_{\lambda})}{2}$};
        \fill[red!80!black] (-1.5,1.8) circle (0.05);
        \node[red!80!black] at (-1.65,2) {\scriptsize$(x^{k})^{\lambda}$};
        \draw[red!70!black,thick] (-1.5,1.8) circle (0.6);
        \draw[red!70!black,thick,dashed] (-1.5,1.8)--(-0.9,1.8);
        \node[red!70!black] at (-1.2,1.6) {\scriptsize$\delta$};
        \draw[red!70!black,thin,<-] (-2.5,1.25) arc(270:360:0.9 and 0.3);
        \node[red!70!black] at (-3.2,1.6) {$B_{\delta}((x^{k})^{\lambda})$};
        \node[black] at (4,-4) {\Large$\Sigma_{\lambda}$};
          \node[below=1cm, align=center] at (0,-6)
            {\fontsize{12}{12}\selectfont Figure~4. The domains $\tilde D, \tilde D^\lambda, \tilde D_R$ and $B_\delta(x^k)$.};
 \end{tikzpicture}
  \end{center}

Let
\[
\phi(x) = a \bigl(\delta^2 - |x - x^0|^2\bigr)_+^s, \quad 
\phi_\lambda(x) = a \bigl(\delta^2 - |x^\lambda - x^0|^2\bigr)_+^s.
\]
Denote
$$
\varphi (x)=\phi(x)-\phi_\lambda (x).
$$
where   $a$  is a positive constant satisfying
\[
(-\Delta)^s \phi(x) = 1, \quad x \in B_\delta(x^0).
\]

It is then clear that $\varphi(x)$ is antisymmetric with respect to the plane $T_\lambda$.

We now define $\tilde D_\lambda$ as the reflection of $\tilde D$ with respect to the plane $T_\lambda$, and
$$
\underline w(x)= w(x) \chi_{\tilde D\cup \tilde D_\lambda} +\varepsilon \varphi (x),
$$
where $\varepsilon$ is a positive constant to be determined later.

Since $w(y) \geq \tilde C>0$ in $\tilde C$, for $x\in B_\delta (x^0)$, a direct calculation yields
\begin{eqnarray*}
\begin{aligned}
(-\Delta)^s \underline w(x)=& (-\Delta)^s (\chi_{\tilde D\cup \tilde D_\lambda} w(x))+\varepsilon  (-\Delta)^s \varphi(x)\\
=& C_{n, s}  \int_{\tilde D \cup {\tilde D_\lambda}}  \frac{-w(y)}{|x-y|^{n+2s}} dy + C_1 \varepsilon-\varepsilon (-\Delta)^s \phi(x^\lambda)\\
\leq & -\tilde C C_{n, s}  \int_{\tilde D}   \left( \frac{1}{|x-y|^{n+2s}} -\frac{1}{|x-y^\lambda|^{n+2s}} \right)dy + C_1 \varepsilon\\
&+\varepsilon C_{n, s} \int_{B_\delta((x^k)^\lambda)} \frac{1}{|x-y|^{n+2s}}dy \\
\leq & - C_2 +C_3\varepsilon.
\end{aligned}
\end{eqnarray*}
Let $\varepsilon=\frac{C_2}{2C_3}$. Then 
\begin{equation}\label{2uderlineu}
(-\Delta)^s \underline w(x) \leq -\frac{C_2}{2},\quad x\in B_\delta(x^0).
\end{equation}

Noting that $c(x)$ is bounded from below, since $w$ is uniformly continuous,
there exist constants $\tilde \delta_1, \tilde \delta_2>0$ independent of $x^0$
such that 
\begin{equation}\label{2estimateu}
(-\Delta)^s w(x) = c(x) w(x) \geq -\frac{C_2}{2}, \quad\mbox{if}\,\,w(x^0)<\tilde \delta_2\,\,\mbox{and}\,\,  x\in B_{\tilde\delta_1}(x^0)\subset \tilde D_R.
\end{equation}

Taking $\delta=\tilde\delta_1$, and combining  \eqref{2uderlineu} with \eqref{2estimateu}, we deduce 
\begin{equation}\label{2comu}
(-\Delta)^s w(x) \geq\; (-\Delta)^s \underline w(x), \quad\mbox{if}\,\,w(x^0)<\tilde \delta_2\,\,\mbox{and}\,\,  x\in B_{\tilde\delta_1}(x^0)\subset \tilde D_R.
\end{equation}
which implies that $\underline w(x)$ is a sub-solution of $w(x)$ in $B_{\tilde\delta_1}(x^0)\subset \tilde D_R$.

Since $ w(x)$ is nonnegative in $\Sigma_\lambda$, it follows that
$$
\underline w(x)
= \chi_{\tilde D \cup \tilde D_\lambda} w(x) + \frac{C_2}{2C_3} \varphi(x)
\;\leq\; w(x),
\quad x \in \Sigma_\lambda \setminus B_{\tilde\delta_1}(x^0).
$$
Therefore, the sub-solution $\underline w(x)$ satisfies
\begin{eqnarray*}
\begin{cases}
(-\Delta)^s \underline w(x) - (-\Delta)^s w(x) \leq 0, & x \in B_{\tilde\delta_1}(x^0), \\[4pt]
\underline w(x) - w(x) \leq 0, & x \in \Sigma_\lambda \setminus B_{\tilde\delta_1}(x^0).
\end{cases}
\end{eqnarray*}
By the maximum principle  for antisymmetric functions associated with the fractional Laplacian, we deduce
$$
\underline w(x) \leq w(x), \quad x \in B_{\tilde\delta_1}(x^0),
$$
and consequently,
$$
w(x) \geq \frac{C_2}{2C_3} \varphi(x),  \quad\mbox{if}\quad u(x^0)<\tilde \delta_2\,\,\mbox{and}\,\, x \in B_{\tilde\delta_1}(x^0).
$$
In particular,  
\begin{equation*}
w(x^0) \geq \frac{C_2}{2C_3} \varphi(x^0)= \frac{C_2}{2C_3} a \tilde \delta_1^{2s},\,\,\,\mbox{if} \,\, \,u(x^0)\leq \tilde\delta_2.
\end{equation*}

Define
\[
\tilde\delta_0=\min\left\{\frac{C_2}{2C_3} a \tilde \delta_1^{2s},\,\tilde \delta_2\right\}.
\]
We can deduce that for any $x^0\in \tilde D_R$, one has
\[
u(x^0)\ge \tilde \delta_0.
\]
It follows that \eqref{eq1.9} holds.
This completes the proof of Theorem~\ref{AE-anti}.
\end{proof}

\section{Monotonicity}

This section is devoted to the proof of monotonicity (Theorem~\ref{th1.1}), based on the boundary regularity (Theorem~\ref{th1.4}), the narrow region principle in unbounded domains (Theorem~\ref{lem:multi-slab-nrp-simple}), and the averaging effects (Theorems~\ref{th1.3} and \ref{AE-anti}).
In addition, we also require the following results on interior regularity and the Harnack inequality for fractional equations.

\begin{lemma}(\cite[Theorem 1.1]{CLWX})\label{le3.1}
Assume that $u\in C_{loc}^{1, 1}(\mathbb{R}^n)\cap \mathcal{L}_{2s}$ is a nonnegative solution of
\begin{equation*}
(-\Delta)^s u(x)=f(x),\,\, x \in B_1(0),
\end{equation*}
where  $f(x)$ is   bounded 
  in $B_1(0)$.
 Then there exist some positive constants $C_1$ and $C_2$ such that
\begin{equation*}
\|u\|_{C^{[2s], \{2s\}}(B_{1/2}(0))}\leq C_1\left(\|f\|_{L^\infty(B_1(0))}+\|u\|_{L^\infty(B_1(0))}\right), \,\, \mbox{if}\,\, s\neq \frac{1}{2},
\end{equation*}
and
\begin{equation*}
\|u\|_{C^{0, \ln L}(B_{1/2}(0))}\leq C_2\left(\|f\|_{L^\infty(B_1(0))}+\|u\|_{L^\infty(B_1(0))}\right), \,\, \mbox{if}\,\, s= \frac{1}{2},
\end{equation*}
where $[2s]$ and $\{2s\}$ denote the integer part and the fractional part of $2s$ respectively.
\end{lemma}

To begin with, we need the following Harnack inequality for  fractional equations.
\begin{lemma} \label{le3.2}(\cite[Lemma 3]{BGQ2019PAMS} or \cite[Theorem 6.1]{BB-SM-1999})
Let $s \in (0, 1),\, x_0\in \mathbb{R}^n,$ and $c(x) \in L^\infty(B_{2R}(x_0)).$ Then there exists a constant $C_0$ depending only on $n, s$ and $R^{2s}\|c\|_{L^\infty(B_{2R}(x_0))}$	 such that for every nonnegative solution of
$$
(-\Delta)^s u(x)=c(x)u(x),\,\, x\in B_{2R}(x_0),
$$
we have
\begin{equation*}
\sup_{B_R(x_0)} u \leq  C_0 \inf_{B_R(x_0)} u.
\end{equation*}
\end{lemma}

\begin{proof}[Proof of Theorem \ref{th1.1}.] First,  according to   equation \eqref{main} with  $f(0)\geq 0$,  and applying  the strong maximum principle for the fractional Laplacian, we obtain $$u(x)>0\ \mb{in}\ \R^n_+.$$ From the definition of $w_\lambda(x)$ and \eqref{main}, it can be readily verified that $w_\lambda(x)$ satisfies
\begin{equation}\label{eq3.1}
\begin{cases}
(-\Delta)^s w_\lambda (x)=C_\lambda(x) w_\lambda(x),& x\in \Omega_\lambda\cap\{w_\lambda<0\},\\
w_\lambda(x) \geq 0, &x\in \Sigma_\lambda \backslash \Omega_\lambda,\\
w_\lambda(x)=-w_\lambda(x^\lambda),& x\in \Sigma_\lambda,
\end{cases}
\end{equation}
Since $u$ is uniformly continuous in  $\mathbb{R}^n_+$ and continuous on $\partial \mathbb{R}^n_+$, we conclude that $u$ is uniformly continuous in  $\overline{\mathbb{R}^n_+}$.

Indeed, if this is not true, then there exist a constant  $\varepsilon_0>0$ and a sequence $x^k=((x^k)', x_n^k)$ and $\bar x^k=((x^k)', 0)$ such that
$$
|x^k-\bar x^k| \to 0\,\, \mbox{as} \,\, k\to \infty,
$$
but
$$
|u(x^k)-u(\bar x^k)| \geq \varepsilon_0.
$$
Then $u(x^k) \geq \varepsilon_0$  due to $u(\bar x^k)=0$. By the continuity of $u,$
there exists a point $y^k$ on the line segment joining the two points  such that
$u(y^k)=\frac{\varepsilon_0}{2}$.

As a consequence, we have
$$
|x^k-y^k| \to 0\,\, \mbox{as} \,\, k\to \infty,
$$
but
$$
|u(x^k)-u(y^k)| \geq \frac{\varepsilon_0}{2}.
$$
This contradicts the fact that $u$ is uniformly continuous in  $\mathbb{R}^n_+$.

Therefore, $u$ is uniformly continuous in $\overline{\mathbb{R}^n_+}$, and for any $\varepsilon>0$ there exists $\delta>0$ such that
\[
 |u(x)-u(y)|<\varepsilon,\quad \mbox{if} \quad |x-y|<\delta. 
\]
Fix $\lambda>0$. For any point $x=(x',x_n)\in\Omega_\lambda$, one can connect $x$ to the boundary point $(x',0)$ by a finite chain of points whose successive distances are less than $\delta$. Consequently,
\[
|u(x)| \le |u(x',0)| + C_\lambda=C_\lambda,
\]
where $C_\lambda>0$ depends only on $\lambda$ and on the modulus of continuity of $u$.
It follows that $u$ is bounded in any slab, that is, $u$ is bounded in $\Omega_\lambda$ for any fixed $\lambda>0$. Therefore, the coefficient function
$$
C_\lambda(x)=\frac{f(u_\lambda(x))-f(u(x))}{u_\lambda(x)-u(x)}
$$
is bounded in $\Omega_\lambda\cap\{w_\lambda<0\}$ for any fixed $\lambda > 0$ since  $f \in C_{\mathrm{loc}}^{0,1}([0, \infty))$.

To establish the strict monotonicity of $u(x)$ with respect to $x_n$, it suffices to show that the antisymmetric function satisfies $w_\lambda(x) > 0$ in $\Omega_\lambda$ for any $\lambda > 0$.

We now divide the proof into two steps.

\vspace{2mm}

{\it{Step 1.}} We start moving the plane $T_\lambda$ upward from $x_n = 0$  along the $x_n$-direction.
First, by the non-negativity and the boundedness of $u$ in any slab $\Omega_{\lambda}$, for any sufficiently small $\lambda > 0$, we have
\[
w_\lambda(x) = u_\lambda(x) - u(x) \geq -C(\lambda), \quad x \in \Omega_\lambda.
\]
Then, applying Theorem \ref{lem:multi-slab-nrp-simple} to the problem \eqref{eq3.1}, we obtain
\begin{equation}\label{eq3.2}
w_\lambda(x) \geq 0, \quad x \in \Omega_\lambda,
\end{equation}
for sufficiently small $\lambda > 0$, where $\Omega_\lambda$ denotes a narrow region.
Inequality \eqref{eq3.2} provides the starting point for continuing to move the plane $T_\lambda$.

\vspace{2mm}

{\it{Step 2.}} In this step, we continue to move the plane $T_\lambda$ upward along the $x_n$-direction as long as \eqref{eq3.2} holds,  until it reaches its limiting position.
Define
\begin{equation}\label{def-l0}
\lambda_0:=\sup \{\lambda \mid w_\mu (x) \geq 0, \,\, x\in \Sigma_\mu\,\, \mbox{for any}\,\, \mu\leq \lambda\}.
\end{equation}
Our goal is to show that
$$
\lambda_0=+\infty.
$$
Assume, for contradiction, that $0 < \lambda_0 < +\infty$. By the definition of $\lambda_0$, there exists a monotone decreasing sequence $\{\lambda_k\}$ such that $\lambda_k \to \lambda_0$ as $k \to \infty$, along which
\begin{equation}\label{eq3.3-1}
\inf_{\Omega_{\lambda_k}} w_{\lambda_k}=\inf_{\Sigma_{\lambda_k}} w_{\lambda_k}:=-m_k<0.
\end{equation}
We begin by proving that
\begin{equation}\label{eq3.4-1}
m_k \to 0\,\, \mbox{as}\,\,  k \to \infty.
\end{equation}

In fact, if \eqref{eq3.4-1} is not true, there exists a positive constant $M>0$ (independent of $k$) such that
$$
\inf_{\Sigma_{\lambda_k}} w_{\lambda_k}<-M<0.
$$
It follows that for any fixed $k$, there exists $z^k \in \Omega_{\lambda_k}$ such that
\begin{equation}\label{eq3.5-1}
 w_{\lambda_k}(z^k)\leq -M<0,
\end{equation}

If $z^k \in \Sigma_{\lambda_k}\backslash \Sigma_{\lambda_0}$, then $|z^k-(z^k)^{\lambda_k}| \to 0$ as $k\to \infty$, and hence, by the uniform continuity of $u$, we obtain
$$
w_{\lambda_k}(z^k)={u_{\lambda_k}(z^k) -u(z^k) \to 0},
\,\, \mbox{as}\,\, k\to \infty,
$$
which contradicts \eqref{eq3.5-1}.

If $z^k\in \Omega_{\lambda_0}$,  since $u$ is uniformly continuous, then
$$
w_{\lambda_k}(z^k)-w_{\lambda_0}(z^k)={u_{\lambda_k}(z^k)-u_{\lambda_0}(z^k) \to 0,}
\,\, \mbox{as}\,\, k\to \infty.
$$
Moreover, noting that $w_{\lambda_0}(z^k) \geq 0$ and using \eqref{eq3.5-1}, we  have
$$w_{\lambda_k}(z^k)-w_{\lambda_0}(z^k)\leq -M,$$
which leads to a contradiction. Hence, \eqref{eq3.4-1} holds.

Combining \eqref{eq3.3-1} and \eqref{eq3.4-1}, we conclude that there exists a point $x^k \in \Omega_{\lambda_k}$ such that
\begin{equation}\label{ctdct-1}
w_{\lambda_k}(x^k) \leq -m_k + m_k^2 < 0
\end{equation}
for all sufficiently large $k$.

(i)
We first claim that
\begin{equation}\label{eq3.6}
\delta_k:=dist (x^k, T_{\lambda_k})=\lambda_k-x_n^k \,\, \mbox{is bounded from zero for sufficiently large }\,\, k.
\end{equation}

If \eqref{eq3.6} is not true, then $\delta_k \to 0$ as $k \to \infty$.
Let $\eta(x) \in C_0^\infty(B_1(0))$ be a sequence of smooth cut-off functions
 satisfying
\begin{equation}\label{eqeta}
0\leq \eta(x) \leq 1 \,\, \mbox{in}\,\, B_{1}(0) \,\,\mbox{and}\,\, \eta (x) \equiv 1 \,\,\mbox{in}\,\, B_{{1}/{2}}(0).
\end{equation}
Define
\begin{equation*}
\eta_k(x):=\eta\left(\frac{x-x^k}{\delta_k} \right)\in C_0^\infty \left(B_{\delta_k} (x^k) \right)
\end{equation*}
and
$$
h_k(x)=w_{\lambda_k}(x)-m_k^2 \eta_k(x).
$$
Then $h_k(x)$ satisfies
$$
h_k(x^k)=w_{\lambda_k}(x^k)-m_k^2 \eta_k(x^k)\leq -m_k+m_k^2-m_k^2=-m_k,
$$
and
$$
h_k(x)=w_{\lambda_k}(x)\geq -m_k,\,\, x\in \Sigma_{\lambda_k}\backslash B_{\delta_k} (x^k)
$$
due to \eqref{eq3.3-1}.
Therefore, there exists a point $\bar x^k \in B_{\delta_k} (x^k)$ such that
$$
-m_k-m_k^2\leq h_k(\bar x^k )=\inf_{\Sigma_{\lambda_k}}h_k(x) \leq  -m_k.
$$
Then, it follows from the definition of $h_k$ that
$$
-m_k\leq w_{\lambda_k}(\bar x^k)  \leq  -m_k+m_k^2<0
$$
for sufficiently large $k$.
Then, at this minimum point $\bar x^k$ of $h_k(x)$, it follows from the definition of the fractional Laplacian that

\begin{equation}\label{eq3.7}
\begin{aligned}
&(-\Delta)^s h_k(\bar x^k)\\
=&C_{n, s} P.V. \int_{\Sigma_{\lambda_k}} \frac{h_k(\bar x^k)-h_k(y)}{|\bar x^k -y|^{n+2s}} dy +C_{n, s} \int_{\Sigma_{\lambda_k}} \frac{h_k(\bar x^k)-h_k(y^{\lambda_k})}{|\bar x^k -y^{\lambda_k}|^{n+2s}} dy \\
\leq &C_{n, s} \int_{\Sigma_{\lambda_k}} \frac{2h_k(\bar x^k)-h_k(y)-h_k(y^{\lambda_k})}{|\bar x^k -y^{\lambda_k}|^{n+2s}} dy \\
\leq & C_{n, s} \int_{\Sigma_{\lambda_k}} \frac{2h_k(\bar x^k)+m_k^2(\eta_k(y)+\eta_k(y^{\lambda_k}))}{|\bar x^k -y^{\lambda_k}|^{n+2s}} dy \\
\leq &\frac{C_1(h_k (\bar x^k)+m_k^2) }{\delta_k^{2s}}\\
\leq & \frac{C_1(-m_k+m_k^2)}{\delta_k^{2s}}.
\end{aligned}
\end{equation}
Here, we have used the fact that $|\bar x^k - y| \leq |\bar x^k - y^\lambda|$ in $\Sigma_{\lambda_k}$ and that $h_k(\bar x^k) \leq -m_k$.

On the other hand, from the first equation in \eqref{eq3.1}, we have
\[
(-\Delta)^s h_k(\bar x^k) = (-\Delta)^s w_{\lambda_k}(\bar x^k) - m_k^2 (-\Delta)^s \eta_k(\bar x^k) \geq -C_2 m_k - \frac{C_3 m_k^2}{\delta_k^{2s}}.
\]

Combining the above estimate with \eqref{eq3.7}, we obtain
\[
-C_2 m_k - \frac{C_3 m_k^2}{\delta_k^{2s}} \leq \frac{C_1(-m_k + m_k^2)}{\delta_k^{2s}}.
\]

Multiplying both sides by $-\frac{\delta_k^{2s}}{m_k}$ and noting that $\lim_{k \to \infty} m_k = 0$ and $\lim_{k \to \infty} \delta_k = 0$, we deduce
\[
0 < C_1 \leq (C_1 + C_3) m_k + C_2 \delta_k^{2s} \to 0 \quad \text{as } k \to \infty,
\]
which is a contradiction.

Therefore, we conclude that $\delta_k$ is bounded away from zero for all sufficiently large $k$.
\vspace{0.5cm}

Next we  claim that
\begin{equation}\label{eq3.66}
\epsilon_k:=\frac{x_n^k}{2}\,\ \mbox{is bounded away from zero for sufficiently large }\,\, k.
\end{equation}
If \eqref{eq3.66} does not hold, then $\epsilon_k \to 0$ as $k \to \infty$. Let
$$\mu_k(x) = \eta\Big(\frac{x - x^k}{\epsilon_k}\Big) \in C_0^\infty(B_{\epsilon_k}(x^k))$$
 be a sequence of smooth cut-off functions
 satisfying
 $$0\leq \mu_k(x)\leq 1 \,\,\mbox{in} \,\, B_{\epsilon_k} (x^k)\,\, \mbox{and}\,\,
 \mu_k(x) \equiv 1 \,\,\mbox{in} \,\, B_{\epsilon_k/2} (x^k),
 $$
 where $\eta$ is defined as in \eqref{eqeta}.

 Define
$$
g_k(x)=w_{\lambda_k}(x)-m_k^2 \mu_k(x).
$$
By an argument analogous to that used for $h_k$, the function $g_k(x)$ satisfies
\begin{equation}\label{key-est}
\begin{cases}
g_k(x^k) \leq -m_k, & \\
g_k(x) = w_{\lambda_k}(x) \geq -m_k, & x \in \Sigma_{\lambda_k} \setminus B_{\epsilon_k}(x^k),
\end{cases}
\end{equation}
and there exists a point $\hat x^k \in B_{\epsilon_k}(x^k) \subset \Omega_{\lambda_k}$ (see Figure 5) such that
\begin{equation}\label{key-est1}
-m_k - m_k^2 \leq g_k(\hat x^k) = \inf_{\Sigma_{\lambda_k}} g_k(x) \leq -m_k.
\end{equation}

It then follows from the definition of $g_k$ that
\[
-m_k \leq w_{\lambda_k}(\hat x^k) \leq -m_k + m_k^2 < 0
\]
for all sufficiently large $k$.

Finally, at this minimum point $\hat x^k$ of $g_k$, we estimate $(-\Delta)^s g_k(\hat x^k)$ as follows:
\begin{equation}\label{estimategk}
\begin{aligned}
(-\Delta)^s g_k(\hat x^k)
&=C_{n, s} P.V. \int_{\Sigma_{\lambda_k}} \frac{g_k(\hat x^k)-g_k(y)}{|\hat x^k -y|^{n+2s}} dy +C_{n, s} \int_{\Sigma_{\lambda_k}} \frac{g_k(\hat x^k)-g_k(y^{\lambda_k})}{|\hat x^k -y^{\lambda_k}|^{n+2s}} dy \\
&=C_{n, s} P.V. \int_{\Sigma_{\lambda_k}} \left(g_k(\hat x^k)-g_k(y)\right)\left(\frac{1}{|\hat x^k -y|^{n+2s}}-\frac{1}{|\hat x^k -y^{\lambda_k}|^{n+2s}}\right) dy\\& \quad\quad+C_{n, s} \int_{\Sigma_{\lambda_k}} \frac{2g_k(\hat x^k)+m_k^2(\mu_k(y)+\mu_k(y^{\lambda_k}))}{|\hat x^k -y^{\lambda_k}|^{n+2s}} dy\\
&\leq  C_0  C_{n,s}  (g_k(\hat x^k)+m_k^2)\int_{B_{\epsilon_k}(\hat x^k_{-})} \frac{1}{|\hat x^k -y|^{n+2s}} dy
+C_{n, s} \int_{\Sigma_{\lambda_k}} \frac{-2m_k+2m_k^2}{|\hat x^k -y^{\lambda_k}|^{n+2s}} dy\\
&\leq C_0 C_{n,s}  (-m_k+m_k^2)\int_{B_{\epsilon_k}(\hat x^k_{-})} \frac{1}{|\hat x^k -y|^{n+2s}} dy\\
&\leq \frac{C(-m_k+m_k^2)}{\epsilon_k^{2s}},
\end{aligned}
\end{equation}
where $\hat x^k_{-}$ denotes the reflection of $\hat x^k$ about the boundary $x_n = 0$.
In the first inequality of \eqref{estimategk}, we have used \eqref{key-est}, \eqref{key-est1}, and the fact that for all $y \in B_{\epsilon_k}(\hat x^k_{-}) \subset \mathbb{R}^n_{-}$, when $k$ is sufficiently large,
\[
 |\hat x^k - y| << |\hat x^k - y^{\lambda_k}|,
\]
which implies that there exists a constant $C_0 > 0$ such that
\[
\frac{1}{|\hat x^k - y|^{n+2s}} - \frac{1}{|\hat x^k - y^{\lambda_k}|^{n+2s}} \geq \frac{C_0}{|\hat x^k - y|^{n+2s}}.
\]

Then, by an argument analogous to that for $h_k$ and using the equation in \eqref{eq3.1}, we obtain a contradiction, which verifies \eqref{eq3.66}.

Consequently, it follows from \eqref{eq3.6} and \eqref{eq3.66} that there exists a positive constant $0 < \delta_0 < \lambda_0$ such that
\[
x^k \in \Sigma_{\lambda_0 - \delta_0} \setminus \Sigma_{\delta_0} \quad \text{for all sufficiently large } k.
\]

\begin{center}
\begin{tikzpicture}[scale=1.3]
        \fill[blue!3] (-3,0) rectangle (2.5,3.5);
        \draw[blue,thin,dashed,-] (-1.8,-0.65)--(-1.8,0.65);
        \draw[blue,thick] (-1.8,-0.65) circle (0.5);
        \fill[blue] (-1.8,-0.65) circle(0.05);%
        \node[blue] at (-1.5,-0.6) {$\hat{x}_{-}^{k}$};%
        \node[blue] at (-3,-0.7) {$B_{\epsilon_{k}}(\hat{x}_{-}^{k})$};
        \draw[blue,thick] (-1.8,0.65) circle (0.5);
        \fill[blue] (-1.8,0.65) circle(0.05);
        \node[blue] at (-1.5,0.58) {$\hat{x}^{k}$};
        \node[blue] at (-3,0.7) {$B_{\epsilon_{k}}(\hat{x}^{k})$};
        \draw[red!80!black,dashed] (-1.6,0.9)--(-1.3,1.4);
        \node[red!80!black] at (-1.6,1.3) {$\epsilon_{k}$};%
        \fill[red!80!black] (-1.6,0.9) circle(0.05);
        \node[red!80!black] at (-1.4,0.9) {$x^{k}$};
           \draw[red!80!black,thick] (-1.6,1) circle (0.5);
           \node[red!80!black] at (-0.5,1.2) {$B_{\epsilon_{k}}(x^{k})$};
       \draw[black,thick,-] (-3,3.5)--(2.5,3.5);
    	\draw[black,thick,->] (-3.2,0)--(3,0);
    	\node[black] at (3.15,0) {$x'$};
    	\draw[black,thick,->] (0,-2.5)--(0,5);
    	\node[black] at (0.25,5) {$x_{n}$};
    	\node[black] at (0.25,-0.25) {$\textit{O}$};
        \fill[black] (0,0) circle(0.04);
        \node[black] at (2.75,4) {$T_{\lambda_{k}}$};
        \node[blue!40!black] at (1.5,1.5) {\Large$\Omega_{\lambda_{k}}$};
          \node[below=1cm, align=center] at (0,-2.5)
            {\fontsize{12}{12}\selectfont Figure~5. The domains $B_{\varepsilon_k}(x^k), B_{\varepsilon_k}(\hat x^k)$ and $B_{\varepsilon_k}(\hat x^k_-)$.};
 \end{tikzpicture}
\end{center}


(ii)
To proceed with the proof, we further claim that
\begin{equation}\label{eq3.9}
u(x^k) \to 0 \quad \text{as } k \to \infty.
\end{equation}

Indeed, if this is not true, there would exist a positive constant $C_4 > 0$ and a subsequence of $\{x^k\}$, which we still denote by $\{x^k\}$, such that
\[
u(x^k) \geq 2C_4 > 0.
\]
It then follows from the uniform continuity of $u$ that there exists a small constant $r_1 > 0$ such that $B_{r_1}(x^k) \subset \Omega_{\lambda_0}$ (see Figure 6) and
\[
u(x) \geq {C_4} > 0, \quad x \in B_{r_1}(x^k).
\]

By the averaging effect (Theorem \ref{th1.3}), there exist small constants $r_2 > 0$ and $C_5 > 0$ such that   $B_{r_2}((x^k)', 2\lambda_0) \subset \Omega_{2\lambda_0 + \delta_0}$ and
\[
u(x) \geq C_5 > 0, \quad x \in ~B_{r_2}((x^k)', 2\lambda_0),
\]
in particular,
$$
u((x^k)',2\lambda_0)\geq {C_5}>0.
$$
It follows from the definition of $w_\lambda$ and $u((x^k)', 0)=0$ that
$$
w_{\lambda_0}((x^k)', 0)\geq {C_5}>0.
$$
By  the boundary continuity of $u$ (Theorem \ref{th1.4}),  there exist small  constants $r_3 > 0$  and $C_6>0$ such that
$$
w_{\lambda_0}(x)\geq C_6>0, \quad x \in B_{r_3}((x^k)', 0),
$$

Noting that $w_{\lambda_0}$ is nonnegative in $\Sigma_{\lambda_0}$ and satisfies
\[
(-\Delta)^s w_{\lambda_0}(x) = C_{\lambda_0}(x) w_{\lambda_0}(x), \quad x \in \Sigma_{\lambda_0},
\]
where
\[
C_{\lambda_0}(x) = \frac{f(u_{\lambda_0}(x)) - f(u(x))}{u_{\lambda_0}(x) - u(x)},
\]
we can apply the averaging effect (Theorem \ref{AE-anti}) to $w_{\lambda_0}(x)$ once again.
This yields the existence of small positive constants $r_4$ and $C_7$ such that
\[
w_{\lambda_0}(x) \geq C_7 > 0, \quad x \in B_{r_4}(x^k).
\]
By $\lambda_k \to \lambda_0$ and the  uniform continuity of $u$, we have
\[
w_{\lambda_k}(x) \geq \frac{C_7}{2} > 0, \quad x \in B_{r_4}(x^k),
\]
which contradicts the fact that $w_{\lambda_k}(x^k) \to 0$ as $k \to \infty$.
Hence, \eqref{eq3.9} is verified.


\begin{center}
\begin{tikzpicture}[scale=1]
    	\draw[black,thick,->] (-4.3,0)--(4.3,0);
        \node[black] at (4.5,0){$x'$};
    	\draw[black,thick,->] (0,-1.5)--(0,8);
    	\node[black] at (0.25,8.2) {$x_{n}$};
        \draw[black,thick,dashed] (-3.5,3)--(3.5,3);
        \node[black] at (4.0,3){$T_{\lambda_{0}}$};
        \draw[black,thick,dashed] (-3.5,6)--(3.5,6);
        \node[black] at (4,6){${\color{blue}T_{2\lambda_{0}}}$};
        \draw[red!80!black,thick] (-1.5,0) circle (0.9);
        \fill[blue,thick] (-1.5,0) circle (0.05);
        \node[blue] at (-1.5,-0.25) {\scriptsize$((x^{k})',0)$};
         \node[blue] at (-1.5,5.75) {\scriptsize$((x^{k})',2\lambda_0)$};
        \draw[red!80!black,thick,dashed] (-1.5,0)--(-0.8,0.6);
        \node[red!80!black] at (-1.45,0.45) {$r_{3}$};
        \draw[blue,thick] (-1.5,2) circle (0.9);
        \fill[blue] (-1.5,2) circle (0.05);
        \node[blue] at (-1.4,1.8) {$x^{k}$};
        \draw[blue,thick,dashed] (-1.5,2)--(-0.7,2.45);
         \draw[red!80!black,thick,dashed] (-1.5,2)--(-2,2);
        \node[blue] at (-1.1,2.5) {$r_{1}$};
        \draw[red!80!black,thin,->] (-1.75,2.05) arc(0:90:0.9 and 0.3);
        \node[red!80!black] at (-3.5,2.4) {$w_{\lambda_0}\geq C_7$};
        \draw[red!80!black,thin,->] (-1.85,0.55) arc(0:90:0.9 and 0.3);
        \node[red!80!black] at (-3.6,0.8) {$w_{\lambda_0}\geq C_6$};
        \draw[blue!80!black,thin,->] (-0.7,2) arc(180:90:1.0 and 0.4);
        \node[blue!80!black] at (1,2.4) {$u\geq C_4$};
         \draw[blue!80!black,thin,->] (-0.7,6.2) arc(180:90:1.0 and 0.4);
         \node[blue!80!black] at (1,6.6) {$u\geq C_5$};
        \draw[blue,thick] (-1.5,6) circle (0.9);
        \node[red!80!black] at (-1.75,1.8) {$r_{4}$};
        \draw[red!80!black,thick] (-1.5,2) circle (0.5);
        \fill[blue] (-1.5,6) circle (0.05);
        \draw[blue,thick,dashed] (-1.5,6)--(-0.8,6.6);
        \node[blue] at (-1.4,6.5) {$r_{2}$};
        \node[black] at (0.25,-0.25) {$\textit{O}$};
        \fill[black] (0,0) circle(0.04);
         \node[below=1cm, align=center] at (0,-1)
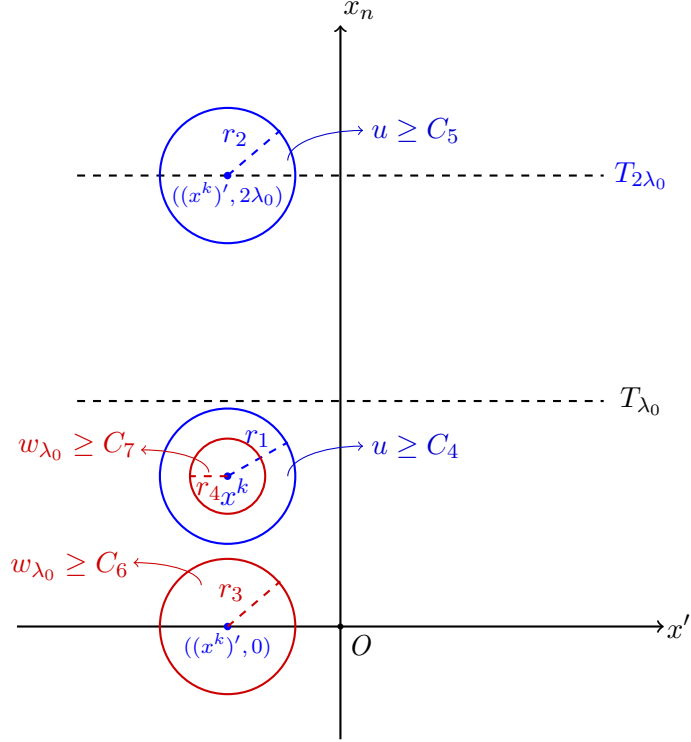
     {\fontsize{12}{12}\selectfont Figure~6. The domains $B_{r_1}(x^k), B_{r_2}((x^k)',2\lambda_0)$ and $B_{r_3}((x^k)',0)$.};     \end{tikzpicture}
 \end{center}

Next, we  employ   perturbation techniques to prove that
\begin{equation}\label{f0}
f(0) = 0.
\end{equation}

Let $\eta(x) \in C_0^\infty(B_1(0))$ be a smooth cut-off function satisfying
\begin{equation*}
\begin{cases}
0\leq \eta(x) \leq 1, & x\in B_1(0),\\
\eta(x)=1,& x\in B_{1/2}(0),\\
\eta (x)=0,& x\in \mathbb{R}^n \backslash B_1(0).
\end{cases}
\end{equation*}
Denote
\[u_k(x):=u(x)-\alpha_k\eta_k(x)\,\,\mbox{ and }\,\, \eta_k(x):=\eta(\frac{x-x^k}{\delta_0})\]
with
\begin{equation}\label{def-afk}
   \alpha_k := u(x^k) \to 0 \ \mbox{as}\ k \to \infty.
\end{equation}

Then $u_k$ satisfies
\[u_k(x^k)=u(x^k)-\alpha_k\eta_k(x^k)=0,\]
and
\[u_k(x)=u(x)\geq 0,\ x \in \mathbb{R}^n\backslash B_{\delta_0}(x^k).\]
Therefore, there exists a sequence of  $\{\tilde{x}_k\}\subset B_{\delta_0}(x^k)\subset \R^n_+$ such that
\begin{equation}\label{inf-uk}-\alpha_k\leq u_k(\tilde{x}_k)=\inf_{x\in \mathbb{R}^n} u_k(x)\leq 0.\end{equation}
equivalently,
\begin{equation}\label{cov-u}0\leq u(\tilde{x}_k)\leq \alpha_k\to 0,\ k\to\infty.\end{equation}

In addition,  it follows from the equation \eqref{main}, the definition of $u_k$, and \eqref{inf-uk} that
\begin{eqnarray*}
\begin{aligned}
f(u(\tilde{x}_k))&=(-\Delta)^s u(\tilde{x}_k)\\
&=(-\Delta)^s u_k(\tilde{x}_k)+\alpha_k(-\Delta)^s \eta_k(\tilde{x}_k)\\
&\leq \alpha_k(-\Delta)^s \eta_k(\tilde{x}_k)\\
&\leq C\alpha_k,
\end{aligned}
\end{eqnarray*}
then let $k\to \infty$, by \eqref{cov-u} and the fact that $f\in C_{loc}(\mathbb{\R}), f(0)\geq 0$, we derive \eqref{f0}, that is,
\[f(0) =  0.\]

(iii)
Next, we aim at showing that \( w_{\lambda_k}(x^k) > 0 \), which will lead to a contradiction with \eqref{ctdct-1}.

To proceed, we define
 \begin{equation}\label{def-vk}
v_k(x) := \frac{u(x)}{\alpha_k},
\end{equation}
where the normalization constant $ \alpha_k>0 $ is given as in \eqref{def-afk}.

In view of  \eqref{f0},  $v_k(x)$ satisfies
\begin{equation}\label{eq-vk}
\begin{cases}
(-\Delta)^s v_k(x) = \frac{1}{\alpha_k} f(u(x)) = \frac{f(u)-f(0)}{u-0} \frac{u}{\alpha_k} = c(x)v_k(x), & x \in \mathbb{R}^n_+,\\
v_k(x) = 0, & x \in \mathbb{R}^n \backslash \mathbb{R}^n_+,\\
v_k(x) \geq 0, & x \in \mathbb{R}^n_+,
\end{cases}
\end{equation}
where $c(x)$ is uniformly bounded in $\Sigma_\lambda$ for each $\lambda>0$.

From the equation  \eqref{eq-vk}, the Harnack inequality (Lemma \ref{le3.2}) applied to $v_k$  implies
\[
\sup_{{B}_{\lambda_0}((x^k)',\delta_0+\lambda_0)} v_k \leq C \inf_{{B}_{\lambda_0}((x^k)',\delta_0+\lambda_0)} v_k \leq C,
\]
where $C$ is a constant depending on $\lambda_0$ but  independent of $k$. Here we have used the fact that $$x^k \in {B}_{\lambda_0}((x^k)',\delta_0+\lambda_0)\ \mbox{and}\ v_k(x^k)=1.$$
Furthermore, by Lemma~\ref{le3.1}, we obtain that, for any $\gamma \in (0,s)$,\begin{equation}\label{bdd-vk}
    \|v_k\|_{C^\gamma({B}_{\lambda_0-\delta_0/2}((x^k)',\delta_0+\lambda_0))}\leq C.
    \end{equation}

Since $v_k(x^k) = 1$, it follows from the uniform continuity of $u$  and \eqref{bdd-vk}
that  there exist a small constant $r_5 > 0$ and a constant $C_8 > 0$ such that $B_{r_5}(x^k) \subset \Omega_{\lambda_0}$ (see Figure 7) and
\[
v_k(x) \geq C_8 > 0, \quad x \in B_{r_5}(x^k).
\]
We can apply the averaging effect (Theorem \ref{th1.3}) to $v_k(x)$ once again.
It follows that there are small positive constants $r_6 > 0$ and $C_9 > 0$ such that
$$
v_k(x) \geq C_9>0, \,\, x\in B_{r_6}((x^k)', 2\lambda_0).
$$
Denote
$$
\tilde w_{\lambda_0}^{(k)}(x)=v_k(x^{\lambda_0})-v_k(x).
$$
Noting that $v_k(x) \equiv 0$ on $\{x_n=-\delta_0/4\}$ and using the uniform continuity of $u$ together with \eqref{bdd-vk}, we infer that there exist some small constants $r_7 > 0$ and $C_{10} > 0$ such that
\[
\tilde w_{\lambda_0}^{(k)}(x)=v_k(x^{\lambda_0}) \geq C_{10} > 0, \quad x \in B_{r_7}((x^k)', -\delta_0/4)\subset \mathbb{R}^n_-.
\]

Clearly, $\tilde w_{\lambda_0}^{(k)}(x) = \frac{w_{\lambda_0}}{\alpha_k}$ is nonnegative in $\Sigma_{\lambda_0}$ and satisfies
\[
(-\Delta)^s \tilde w_{\lambda_0}^{(k)}(x) = \tilde C_{\lambda_0}(x) \tilde w_{\lambda_0}^{(k)}(x), \quad x \in \Sigma_{\lambda_0},
\]
where
\[
\tilde C_{\lambda_0}(x) = \frac{f(u(x^{\lambda_0})) - f(u(x))}{u(x^{\lambda_0}) - u(x)}.
\]

Therefore, applying the averaging effect (Theorem \ref{AE-anti}) to $\tilde w_{\lambda_0}^{(k)}(x)$ once again, we obtain that there exists  a small constant $C_{11} > 0$ such that
\begin{equation}\label{eq3.11-0}
\tilde w_{\lambda_0}^{(k)}(x)\geq C_{11}>0,\,\, B_{r_5} (x^k).
\end{equation}

 \begin{center}
\begin{tikzpicture}[scale=1]
        \fill[blue!5] (-3,0) rectangle (2.5,3);
    	\draw[black,thick,->] (-3,0)--(2.5,0);
        \node[black] at (2.7,0){$x'$};
    	\draw[black,thick,->] (0,-1.5)--(0,7);
    	\node[black] at (0.3,7) {$x_{n}$};
        \draw[black,thick,-] (-3,3)--(2.5,3);
        \node[black] at (3.2,3){$x_{n}=\lambda_{0}$};
        \draw[black,thick,-] (-3,6)--(2.5,6);
        \node[black] at (3.3,6){$x_{n}=2\lambda_{0}$};
        \draw[red!80!black,thick] (-1.5,-0.5) circle (0.9);%
       \fill[blue] (-1.5,-0.5) circle (0.05);%
        \node[blue] at (-1.5,-0.75) {\scriptsize$((x^{k})',-\frac{\delta_0}{4})$};%
        \draw[red!80!black,thick,dashed] (-1.5,-0.5)--(-0.65,-0.15);
        \node[red!80!black] at (-1.3,-0.15) {$r_{7}$};
        \draw[blue,thick] (-1.5,2) circle (0.9);
        \fill[blue] (-1.5,2) circle (0.05);
        \node[blue] at (-1.5,1.8) {$x^{k}$};
        \draw[blue,thick,dashed] (-1.5,2)--(-0.65,2.35);
        \node[blue] at (-1.2,2.4) {$r_{5}$};
        \draw[blue,thick] (-1.5,6) circle (0.9);
          \draw[red!80!black,thick] (-1.5,2) circle (0.4);
        \fill[blue] (-1.5,6) circle (0.05);
        \draw[blue,thick,dashed] (-1.5,6)--(-0.65,6.35);
        \node[blue] at (-1.3,6.4) {$r_{6}$};
        \fill[black] (0,0) circle(0.04);
         \draw[red!80!black,thin,->] (-1.75,2.05) arc(0:90:0.9 and 0.3);
        \node[red!80!black] at (-3.7,2.4) {$\tilde w_{\lambda_0}^{(k)}\geq C_{11}$};
           \draw[red!80!black,thin,->] (-1.85,-0.45) arc(0:90:0.9 and 0.3);%
        \node[red!80!black] at (-3.7,-0.2) {$\tilde w_{\lambda_0}^{(k)}\geq C_{10}$};%
        \draw[blue!80!black,thin,->] (-0.7,2) arc(180:90:1.0 and 0.4);
        \node[blue!80!black] at (1,2.4) {$v_k\geq C_8$};
         \draw[blue!80!black,thin,->] (-0.7,6.1) arc(180:90:1.0 and 0.4);
         \node[blue!80!black] at (1,6.5) {$v_k\geq C_9$};
        \node[black] at (0.25,-0.25) {$\textit{O}$};
        \node[orange!70!black] at (1.5,1.5) {\Large$\Omega_{\lambda_{0}}$};
        \node[blue] at (-1.5,5.75) {\scriptsize$((x^{k})',2\lambda_{0})$};
          \node[below=1cm, align=center] at (0,-1)
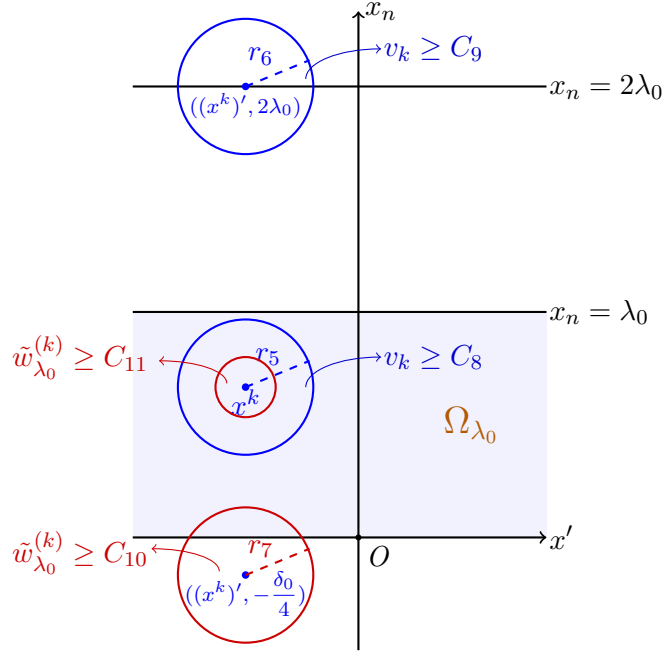
     {\fontsize{12}{12}\selectfont Figure~7. The domains $B_{r_5}(x^k), B_{r_6}((x^k)',2\lambda_0)$ and $B_{r_7}((x^k)',-\frac{\delta_0}{4})$.};
 \end{tikzpicture}
 \end{center}


We are left to show that there exists a sufficiently small constant
 $0<r_0<r_5$ such that \begin{equation}\label{eq3.11}
\tilde w_{\lambda_k}^{(k)}(x):=v_k(x^{\lambda_k})-v_k(x)
\geq C_{11}/2>0, \quad x \in B_{r_0} (x^k),
\end{equation}
for sufficiently large $k$. Consequently,
\[
w_{\lambda_k}(x^k) = \alpha_k \, \tilde w_{\lambda_k}^{(k)}(x^k) > 0,
\]
which contradicts \eqref{ctdct-1}.


Now, take $r_0 <\min\{\frac{ \delta_0}{2},r_5\}$ sufficiently small such that
\[
B_{r_0}(x^k) \subset \Sigma_{\lambda_0-\frac{ \delta_0}{2}}\backslash \Sigma_{\frac{ \delta_0}{2}},
\]
then for each $x \in B_{r_0}(x^k)$, one has
\[
x^{\lambda_0}, \, x^{\lambda_k} \in {B}_{\lambda_0-\delta_0}((x^k)',\delta_0+\lambda_0).
\]
It follows from \eqref{bdd-vk} that
\begin{equation}\label{eq3.17}
\begin{aligned}
\left| \tilde w_{\lambda_0}^{(k)}(x)-\tilde w_{\lambda_k}^{(k)}(x) \right|
=& \left|v_k(x^{\lambda_0}) -v_k(x^{\lambda_k}) \right|\\
\leq & C \left|x^{\lambda_0} -x^{\lambda_k} \right|^\gamma \\
\leq & C |\lambda_0-\lambda_k|^\gamma
\to 0 \quad \text{as } k \to \infty.
\end{aligned}
\end{equation}

Since we already have (see \eqref{eq3.11-0})
\[
\tilde w_{\lambda_0}^{(k)}(x)\geq C_{11}>0, \quad x\in B_{r_5} (x^k),
\]
combining this with \eqref{eq3.17} yields \eqref{eq3.11}, which contradicts \eqref{ctdct-1}.
Therefore, $\lambda_0=\infty$.

{\it{Step 3.}} In this step, we show that $\partial_{x_n}u(x)>0$ for all $ x\in \mathbb R^n_+.$ 

From the definition of $\lambda_0$ and $\lambda_0=\infty,$ we derive
\[
w_\lambda(x)\ge 0 \quad \text{in } \Sigma_\lambda,\qquad \forall\,\lambda>0.
\]
By the strong maximum principle, we further obtain
\[
w_\lambda(x)>0 \quad \text{in } \Sigma_\lambda,\qquad \forall\,\lambda>0.
\]

Applying the Hopf  lemma to $w_\lambda$ at the flat boundary
$T_\lambda=\{x_n=\lambda\}$, we deduce that
\[
\partial_{\nu} w_\lambda(x)>0 \quad \text{for all } x\in T_\lambda,
\]
where $\nu=\mathbf e_n$ denotes the outward unit normal vector to
$\Sigma_\lambda=\{x_n<\lambda\}$.

For $x=(x',\lambda)\in T_\lambda$, the outward normal derivative is given by
\[
\partial_{\nu} w_\lambda(x)
=\lim_{h\to 0}\frac{w_\lambda(x',\lambda+h)-w_\lambda(x',\lambda)}{h}.
\]
Since $w_\lambda(x',\lambda)=0$, we have
\[
w_\lambda(x',\lambda+h)
=u(x',\lambda-h)-u(x',\lambda+h),
\]
and hence
\[
\partial_{\nu} w_\lambda(x)
=\lim_{h\to 0}
\frac{u(x',\lambda-h)-u(x',\lambda+h)}{h}
=-2\,\partial_{x_n}u(x',\lambda).
\]
Consequently,
\[
\partial_{x_n}u(x',\lambda)>0 \quad \text{for all } x' \in \mathbb R^{n-1}.
\]

Since $\lambda>0$ is arbitrary, we conclude that
\[
\partial_{x_n}u(x)>0 \quad \text{for all } x\in \mathbb R^n_+.
\]
This completes the proof of Theorem~\ref{th1.1}.
\end{proof}
\bigbreak
\bigbreak
\bigbreak

\begin{proof}[Proof of Theorem \ref{th1.10}]
Similar to the first step in the proof of Theorem~\ref{th1.1}, we first obtain that
\[
w_{\lambda}\ge 0 \quad \text{in } \Omega_{\lambda},
\]
for all sufficiently small $\lambda>0$.

We then move the plane $T_{\lambda}$ in the $x_n$-direction to its limiting position $T_{\lambda_0}$, where
\begin{equation}\label{def-l0b}
\lambda_0:=\sup \Bigl\{\lambda \,\big|\, w_\mu(x)\ge 0 \ \text{in } \Sigma_\mu \ \text{for any } \mu\le \lambda \Bigr\}.
\end{equation}
Assuming that $u$ is bounded in any slab $\Sigma_\lambda$, we apply a limiting argument to show that $\lambda_0=+\infty$.

Suppose by contradiction that $0<\lambda_0<+\infty$. Then there exists a decreasing sequence $\{\lambda_k\}$ such that
\[
\lambda_k \searrow \lambda_0 \quad \text{as } k\to\infty,
\]
and a sequence $\{z^k\}\subset \Omega_{\lambda_k}$ satisfying
\begin{equation}\label{wzk-neg}
w_{\lambda_k}(z^k)<0.
\end{equation}

By the narrow region principle (Theorem~\ref{lem:multi-slab-nrp-simple}), we may further assume that
\begin{equation}\label{zk}
z^k\in \Sigma_{\lambda_0-\delta}\setminus\Sigma_\delta,
\end{equation}
for some sufficiently small $\delta>0$.

Indeed, otherwise we would have
\[
w_{\lambda_k}(x)\ge 0 \quad \text{in } \Sigma_{\lambda_0-\delta}\setminus\Sigma_\delta.
\]
Since $(\Sigma_{\lambda_k}\setminus\Sigma_{\lambda_0-\delta})\cup\Sigma_\delta$ is a narrow region in $\Sigma_{\lambda_k}$ for small $\delta$ and large $k$, the narrow region principle implies that
\[
w_{\lambda_k}(x)\ge 0 \quad \text{in } \Sigma_{\lambda_k},
\]
which contradicts \eqref{wzk-neg}.

Define the translated functions
\begin{equation}\label{u-trans}
\tilde u_k(x):=u(x'+(z^k)',x_n).
\end{equation}
By \eqref{main}, each $\tilde u_k$ solves
\begin{equation}\label{tuk}
\begin{cases}
(-\Delta)^s \tilde u_k=f(\tilde u_k), & x\in \mathbb{R}^n_+,\\
\tilde u_k=0, & x\in \mathbb{R}^n\backslash \mathbb{R}^n_+,\\
\tilde u_k\ge 0, & x\in \mathbb{R}^n_+.
\end{cases}
\end{equation}

We first claim that
\begin{equation}\label{conv0}
\tilde u_k \to 0 \quad \text{locally uniformly in } \mathbb{R}^n_+.
\end{equation}

Indeed, since $u$ is bounded in any slab $\Sigma_\lambda$ and $f\in C^{0,1}_{\mathrm{loc}}([0,\infty))$, by interior Schauder estimates (Lemma~\ref{le3.1}) and boundary Hölder regularity (Theorem~\ref{th1.4}), we obtain
\begin{equation*}
\|\tilde{u}_k\|_{C^{2s+\varepsilon}(\Sigma_{\lambda}\backslash\Sigma_{\delta})}\leq C\|\tilde{u}_k\|_{L^{\infty}(\Sigma_{\lambda})}\leq C_{\lambda},
\end{equation*}
and
\begin{equation*}
\|\tilde{u}_k\|_{C^{s}(\overline{\Sigma}_{\delta})}\leq C\|\tilde{u}_k\|_{L^{\infty}(\Sigma_{\delta})}\leq C_{\lambda},
\end{equation*}
for some  $\varepsilon>0$.

Passing to a subsequence if necessary, $\tilde u_k$ converges locally uniformly to a function
\[
u_\infty\in C^{2s+\varepsilon}_{\mathrm{loc}}(\mathbb{R}^n_+)\cap C^s(\overline{\mathbb{R}^n_+}),
\]
which satisfies
\[
\begin{cases}
(-\Delta)^s u_\infty=f(u_\infty), & x\in \mathbb{R}^n_+,\\
u_\infty=0, & x\in \mathbb{R}^n_-,\\
u_\infty\ge 0, & x\in \mathbb{R}^n_+.
\end{cases}
\]

Next, we show that
\begin{equation}\label{eqv0}
u_\infty\equiv 0 \quad \text{in } \mathbb{R}^n_+.
\end{equation}

Let
\[
w_{\infty,\lambda_0}:=u_{\infty,\lambda_0}-u_\infty,
\qquad
\tilde z^k:=(0',z^k_n).
\]
By \eqref{zk} and \eqref{u-trans}, the sequence $\{\tilde z^k\}$ is bounded in $\Sigma_{2\lambda_0}$, and
\begin{equation}\label{eqv}
\tilde u_k(\tilde z^k)=u(z^k).
\end{equation}
Thus, up to a subsequence, there exists $\tilde z_n \in (\delta, \lambda_0-\delta)$ such that
\begin{equation*}
\tilde{z}^k \to \tilde z:=(0',\tilde{z}_{n}),\,\, \mbox{as}\,\, k\to \infty.
\end{equation*}

Using \eqref{wzk-neg}, \eqref{eqv}, and the uniform continuity of $u$, we obtain
\begin{equation*}
\begin{aligned}
w_{\infty,\lambda_0}(\tilde{z})=& \lim\limits_{k\to\infty}(\tilde u_{k,\lambda_0}-\tilde u_k)(\tilde z) \\=&
\lim\limits_{k\to\infty}(\tilde u_{k,\lambda_0}-\tilde u_k)(\tilde z^k)
\\
=&\lim\limits_{k\to\infty}( u_{\lambda_0}- u)( z^k) \\
=&\lim\limits_{k\to\infty}( u_{\lambda_k}- u)( z^k) \leq 0.\\
\end{aligned}
\end{equation*}

On the other hand, by the definition of $\lambda_0$ in \eqref{def-l0b} and the translation \eqref{u-trans}, we have
\[
\tilde u_{k,\lambda_0}\ge \tilde u_k \quad \text{in } \Sigma_{\lambda_0}.
\]
Passing to the limit as $k\to\infty$, it follows that
\[
w_{\infty,\lambda_0}\ge 0 \quad \text{in } \Sigma_{\lambda_0}.
\]
Therefore,
\begin{equation*}
w_{\infty,\lambda_0}(\tilde z)=0.
\end{equation*}
By the strong maximum principle for antisymmetric functions,
\[
w_{\infty,\lambda_0}\equiv 0 \quad \text{in } \Sigma_{\lambda_0}.
\]
It follows that
\[
u_{\infty}(x)=0 \quad \text{for all } x\in T_{2\lambda_0}.
\]
Applying the strong maximum principle to $u_{\infty}$ once again, we obtain
\[
u_{\infty}\equiv 0 \quad \text{in } \mathbb{R}^n_+.
\]
Therefore, \eqref{eqv0} and \eqref{conv0} are verified, which in turn yields
\[
f(0)=0.
\]

Now, for each $k\ge 1$, define
\[
\tilde v_k(x):=\frac{\tilde u_k(x)}{\tilde u_k(\tilde z)}.
\]
Then, by \eqref{tuk}, $\tilde v_k$ satisfies
\begin{equation*}
\begin{cases}
(-\Delta)^s \tilde v_k(x)
= \dfrac{f(\tilde u_k(x))-f(0)}{\tilde u_k(x)}\,\tilde v_k(x)
=: c_k(x)\tilde v_k(x), & x\in \mathbb{R}^n_+,\\[1mm]
\tilde v_k(x)=0, & x\in \mathbb{R}^n_-,\\
\tilde v_k(x)\ge 0, & x\in \mathbb{R}^n_+ .
\end{cases}
\end{equation*}
Here, $c_k$ is locally uniformly bounded in $\mathbb{R}^n_+$.

According to the definition of $\lambda_0$, the functions $\tilde v_k$ are monotone increasing in the $x_n$-direction in $\Sigma_{\lambda_0}$.
Applying the Harnack inequality (Lemma~\ref{le3.2}) to $\tilde v_k$, we obtain
\[
\sup_{B'_r(0')\times [0,T]} \tilde v_k
\le C\,\tilde v_k(\tilde z)=C,
\]
for any $r>0$ and $T>\tilde z_n$, where $C$ depends only on $r$ and $T$, but is independent of $k$.

Furthermore, by Lemma~\ref{le3.1} and Theorem~\ref{th1.4}, we have
\begin{equation}\label{I-bdd-vkt}
\|\tilde v_k\|_{C^{2s+\varepsilon}_{\mathrm{loc}}(\mathbb{R}^n_+)}\le C,
\end{equation}
and
\begin{equation*}
\|\tilde v_k\|_{C^{s}(B'_r(0')\times [0,h])}\le C_{r,h}.
\end{equation*}
Letting $k\to\infty$, it follows that $\tilde v_k$ converges locally uniformly in $\mathbb{R}^n_+$ to some
\[
v_\infty\in C^{2s+\varepsilon}_{\mathrm{loc}}(\mathbb{R}^n_+)\cap C(\overline{\mathbb{R}^n_+}),
\]
which solves
\begin{equation*}
\begin{cases}
(-\Delta)^s v_\infty(x)=c_\infty(x)v_\infty(x), & x\in \mathbb{R}^n_+,\\
v_\infty(x)=0, & x\in \mathbb{R}^n\backslash \mathbb{R}^n_+,\\
v_\infty(x)\ge 0, & x\in \mathbb{R}^n_+ .
\end{cases}
\end{equation*}

As in the argument for $u_\infty$ above, it suffices to show that
\begin{equation}\label{eqvv}
v_{\infty,\lambda_0}(\tilde z)=v_\infty(\tilde z).
\end{equation}
Clearly, by the definition of $\lambda_0$,
\begin{equation}\label{gev}
v_{\infty,\lambda_0}(\tilde z)\ge v_\infty(\tilde z).
\end{equation}
Moreover, using \eqref{wzk-neg}, \eqref{eqv}, and the local uniform continuity of $\tilde v_k$ implied by \eqref{I-bdd-vkt}, we obtain
\begin{equation*}
\begin{aligned}
\lim_{k\to\infty}(v_{\infty,\lambda_k}-v_\infty)(\tilde z^k)
&=\lim_{k\to\infty}(\tilde v_{k,\lambda_k}-\tilde v_k)(\tilde z^k) \\
&=\lim_{k\to\infty}\frac{(\tilde u_{k,\lambda_k}-\tilde u_k)(\tilde z^k)}{\tilde u_k(\tilde z)} \\
&=\lim_{k\to\infty}\frac{(u_{\lambda_k}-u)(z^k)}{\tilde u_k(\tilde z)} \\
&=\lim_{k\to\infty}\frac{w_{\lambda_k}(z^k)}{\tilde u_k(\tilde z)}\le 0 .
\end{aligned}
\end{equation*}
Therefore,
\begin{equation*}
\begin{aligned}
v_{\infty,\lambda_0}(\tilde z)-v_\infty(\tilde z)
&=\lim_{k\to\infty}\bigl[v_{\infty,\lambda_0}(\tilde z)-v_{\infty,\lambda_k}(\tilde z^k)\bigr]
 +\lim_{k\to\infty}(v_{\infty,\lambda_k}-v_\infty)(\tilde z^k) \\
&\le \lim_{k\to\infty}\bigl[v_{\infty,\lambda_0}(\tilde z)-v_{\infty,\lambda_k}(\tilde z^k)\bigr]=0.
\end{aligned}
\end{equation*}
Combining this with \eqref{gev}, we obtain \eqref{eqvv}.

Consequently, applying the strong maximum principle twice yields
\[
v_\infty\equiv 0 \quad \text{in } \mathbb{R}^n,
\]
which contradicts the fact that
\[
v_\infty(\tilde z)=1.
\]
Therefore, $\lambda_0=\infty$. By an argument analogous to that used in the proof of
Theorem~\ref{th1.1}, we conclude that
\[
\partial_{x_n}u(x)>0 \quad \text{for all } x\in \mathbb R^n_+.
\]
This completes the proof of Theorem~\ref{th1.10}.
\end{proof}

\textbf{Acknowledgments}
 W. Chen is
 partially supported by MPS Simons Foundation 847690 and NSFC (Grant No. 12571225).
 
 Y. Guo is partially supported by the National Natural Science Foundation of China (Grant No.12501145, W2531006, 12250710674 and 12031012), the Natural Science Foundation of Shanghai (No. 25ZR1402207),   the China Postdoctoral Science Foundation (No. 2025T180838 and 2025M773061), the Postdoctoral Fellowship Program of CPSF (No. GZC20252004), and the Institute of Modern Analysis-A Frontier Research Center of Shanghai.

L. Wu is partially supported by National Natural Science Foundation of China (Grant No. 12401133).

\vspace{2mm}

\textbf{Conflict of interest.} The authors do not have any possible conflict of interest.

\vspace{2mm}

\textbf{Data availability statement.}
 Data sharing not applicable to this article as no data sets were generated or analysed during the current study.

\end{document}